\documentclass[reqno]{amsart}
\usepackage{graphicx}
\usepackage{amstext}
\usepackage{amssymb}
\usepackage{amsmath}
\usepackage{color}
\newtheorem{theorem}{Theorem}[section]
\newtheorem{lemma}[theorem]{Lemma}

\newtheorem{proposition}[theorem]{Proposition}
\theoremstyle{definition}

\newtheorem{remark}[theorem]{Remark}
\numberwithin{equation}{section}
\usepackage{hyperref}
\hypersetup{colorlinks,bookmarks=true,linkcolor=blue,citecolor=blue}

\begin{document}

\title{On continuation properties after blow-up time for $L^2$-critical gKdV equations}\def\rightmark{CONTINUATION AFTER BLOW-UP TIME}

\author{Yang Lan}

\address{Department of Mathematics and Computer Science, University of Basel, Spiegelgasse 1, CH-4051 Basel, Switzerland}

\email{yang.lan@unibas.ch}



\begin{abstract}
In this paper, we consider a blow-up solution $u(t)$ (close to the soliton manifold) to the $L^2$-critical gKdV equation $\partial_tu+(u_{xx}+u^5)_x=0$, with finite blow-up time $T<+\infty$. We expect to construct a natural extension of $u(t)$ after the blow-up time. To do this, we consider the solution $u_{\gamma}(t)$ to the saturated $L^2$-critical gKdV equation $\partial_tu+(u_{xx}+u^5-\gamma u|u|^{q-1})_x=0$ with the same initial data, where $\gamma>0$ and $q>5$. A standard argument shows that $u_{\gamma}(t)$ is always global in time. Moreover, for all $t<T$, $u_{\gamma}(t)$ converges to $u(t)$ in $H^1$ as $\gamma\rightarrow0$. We prove in this paper that for all $t\geq T$, $u_{\gamma}(t)\rightarrow v(t)$ as $\gamma\rightarrow0$, in a certain sense. This limiting function $v(t)$ is a weak solution to the unperturbed $L^2$-critical gKdV equations, hence can be viewed as a natural extension of $u(t)$ after the blow-up time.
\end{abstract}

\maketitle

\section{Introduction}
\subsection{Setting of the problem}\label{S11}
In this paper, we consider the $L^2$ critical gKdV equation:
\begin{equation}
\label{CP}\tag{gKdV}
\begin{cases}
\partial_tu+u_{xxx}+(u^5)_x=0,\;(t,x)\in\mathbb{R}\times\mathbb{R},\\
u(0)=u_0\in H^1(\mathbb{R}).
\end{cases}
\end{equation}

From Kato \cite{Ka} and Kenig-Ponce-Vega \cite{KPV}, the Cauchy problem \eqref{CP} is locally well-posed in $H^1$: for all $u_0\in H^1$, there is a unique strong solution $u(t,x)\in \mathcal{C}([0,T),H^1)$ to \eqref{CP}, where $T$ is the maximal lifespan of the solution. Moreover, we have the following blow-up criterion: either $T=+\infty$ or $T<+\infty$ and 
\begin{equation}
\lim_{t\rightarrow T}\|u(t)\|_{H^1}=+\infty.
\end{equation}

As a universal Hamiltonian model, the gKdV equation has two conservation laws, the mass and the energy:
\begin{align}
&M(u(t))=\int|u(t)|^2=M_0,\label{CL1}\\
&E(u(t))=\frac{1}{2}\int|u_x(t)|^2-\frac{1}{6}\int|u(t)|^6=E_0.\label{CL2}
\end{align}

There is a scaling symmetry for \eqref{CP}: for all $\lambda>0$, if $u(t,x)$ is a solution to \eqref{CP}, then so is 
\begin{equation}\label{scaling}
u_{\lambda}(t,x)=\frac{1}{\lambda^{1/2}}u\bigg(\frac{t}{\lambda^3},\frac{x}{\lambda}\bigg).
\end{equation}
The Cauchy problem \eqref{CP} is called $L^2$ critical, since the scaling symmetry \eqref{scaling} leaves the $L^2$ norm of the initial data invariant, $i.e.$ $\|u_{\lambda}(0)\|_{L^2}=\|u(0)\|_{L^2}$ for all $\lambda>0$. 

There is a special class of solutions, called the \emph{soliton solutions} (or solitary waves, traveling waves, $etc.$). They are given by
\begin{equation}
u(t,x)=Q(x-t),
\end{equation}
with
\begin{equation}
Q(x)=\bigg(\frac{3}{\cosh^2(2x)}\bigg)^{\frac{1}{4}}.
\end{equation}
Here the function $Q$ is also called the \emph{ground state}. It is the unique nonnegative, radial solution with exponential decay to the following ODE:
\begin{equation}
Q''-Q+Q^5=0.
\end{equation}

From Weinstein \cite{W1}, the ground state $Q$ satisfies the sharp Gagliardo-Nirenberg's inequality:
\begin{equation}
\forall\, v\in H^1,\quad \int v^6\leq 3\int v_x^2\bigg(\frac{\int v^2}{\int Q^2}\bigg)^2.
\end{equation}
Hence, for all $u_0\in H^1$ with $\|u_0\|_{L^2}<\|Q\|_{L^2}$, the corresponding solution is always global in time and bounded in the energy space $H^1$.

\subsection{Overview of the blow-up dynamics for $L^2$ critical gKdV equations}
\subsubsection{Blow-up dynamics for solutions with slightly supercritical mass}
For $u_0\in H^1$ with $\|u_0\|_{L^2}\geq \|Q\|_{L^2}$, blow-up may occur. In a series of work \cite{MM1,MM2,MM4,MM5,MM3}, \cite{M1}, Martel and Merle obtained the first qualitative results for solution with slightly supercritical mass: $\|Q\|_{L^2}<\|u_0\|_{L^2}<\|Q\|_{L^2}+\alpha^*$, $0<\alpha^*\ll1$. In particular, they proved the existence of solutions blowing up in finite time with negative energy, and the ground state $Q$ is the universal blow-up profile for all $H^1$ blow-up solutions in this regime.

\subsubsection{Classification of the flow near the ground state}
In recent works \cite{MMR1,MMR2}, Martel, Merle and Rapha\"el gave a specific description of the flow near the ground state.

More precisely, for all $0<\alpha_0\ll\alpha^*\ll1$, we let 
\begin{align}
&\mathcal{A}_{\alpha_0}=\bigg\{u_0=Q+\varepsilon_0\;\Big|\|\varepsilon_0\|_{L^2}<\alpha_0,\;\int_{y>0}y^{10}\varepsilon^2_0(y)\,dy<1\bigg\},\\
&\mathcal{T}_{\alpha^*}=\bigg\{u_0\in L^2\Big|\inf_{\lambda_0>0,\,x_0\in \mathbb{R}}\bigg\|u_0(\cdot)-\frac{1}{\lambda_0^{1/2}}Q\bigg(\frac{\cdot-x_0}{\lambda_0}\bigg)\bigg\|_{L^2}<\alpha^*\bigg\}.
\end{align}
Then we have

\begin{theorem}[{Rigidity of the dynamics in $\mathcal{A}_{\alpha_0}$, Theorem 1.1 \& 1.2 in \cite{MMR1}}]\label{T1}
For all $0<\alpha_0\ll\alpha^*\ll1$, and $u_0\in\mathcal{A}_{\alpha_0}$, let $u(t)$ be the corresponding solution to \eqref{CP}, and $0<T\leq+\infty$ be the maximal lifespan. Then one and only one of the following scenarios occurs:
\begin{itemize}
\item{\rm\bf(Blow up):} The solution $u(t)$ blows up in finite time $0<T<+\infty$ with
$$\|u(t)\|_{H^1}=\frac{\ell_0+o(1)}{T-t},\quad \ell_0>0,$$
and for all $0\leq t<T$, $u(t)\in\mathcal{T}_{\alpha^*}$. 

In addition, there exist $\lambda(t)>0$, $x(t)\in \mathbb{R}$ and $u^*\in H^1$, $u^*\not=0$, such that
\begin{equation}
\label{converge}
u(t,\cdot)-\frac{1}{\lambda^{1/2}(t)}Q\bigg(\frac{\cdot-x(t)}{\lambda(t)}\bigg)\rightarrow u^*\;\text{in }L^2,\;\; \text{as } t\rightarrow T,
\end{equation}
with
\begin{equation}
\label{14}
\lim_{t\rightarrow T}\frac{\lambda(t)}{T-t}=\ell_0>0,\quad \lim_{t\rightarrow T}(T-t)x(t)=\ell_0^{-2}.
\end{equation}
\item{\rm\bf(Soliton):} The solution is global, and for all $0\leq t<T=+\infty$, $u(t)\in\mathcal{T}_{\alpha^*}$. In addition, there exist a constant $\lambda_{0}>0$ and a $C^1$ function $x(t)$ such that
\begin{gather*}
\lambda_{0}^{\frac{1}{2}}u(t,\lambda_{0}\cdot+x(t))\rightarrow Q\text{ in }H^{1}_{{\rm loc}},\text{ as }t\rightarrow+\infty,\\
|\lambda_{0}-1|\lesssim\delta(\alpha_0),\quad x(t)\sim\frac{t}{\lambda_{0}^2},\text{ as }t\rightarrow+\infty.
\end{gather*}
\item{\rm\bf(Exit):} For some finite time $0<t^*<T$, $u(t^*)\notin\mathcal{T}_{\alpha^*}$.
\end{itemize}

Moreover, all of the three scenarios are possible to occur and the scenarios (Blow up) and (Exit) are stable by small perturbation in $\mathcal{A}_{\alpha_0}$.
\end{theorem}

\begin{remark}
The decay assumption on the right of the initial data in the definition of $\mathcal{A}_{\alpha_0}$ is important. Indeed, in \cite{MMR3}, Martel, Merle and Rapha\"el constructed $H^1$ blow-up solutions with exotic blow-up rate, where the initial data has a slowly decaying tail on the right.
\end{remark}

\begin{remark}
In \cite{MMR2}, Martel, Merle and Rapha\"el proved the existence and uniqueness of the minimal mass blow-up solution $S(t)$ with $\|S(t)\|_{L^2}=\|Q\|_{L^2}$. They also proved that solutions in the (Exit) case have a universal behavior at the exit time, related to the minimal mass blow up solution $S(t)$. Solutions in this regime are also expected to scatter at $+\infty$. However, it still remains open. 
\end{remark}

\begin{remark}
Recall that in \cite{MMNR}, Martel, Merle, Nakanishi and Rapha\"el proved that the initial data set corresponding to the (Soliton) regime is a codimension one threshold manifold in a small neighborhood of the ground state between the two stable regimes.
\end{remark}

\subsection{The $L^2$-critical gKdV with a saturated perturbation}\label{S13}
Let us recall some results about the saturated problem of $L^2$-critical gKdV:
\begin{equation}\label{CPG}\tag{gKdV\textsubscript{$\gamma$}}
\begin{cases}
\partial_t u +(u_{xx}+u^5-\gamma u|u|^{q-1})_x=0, \quad (t,x)\in[0,T)\times\mathbb{R},\\
u(0,x)=u_0(x)\in H^1(\mathbb{R}),
\end{cases}
\end{equation}
with $q>5$ and $0<\gamma\ll1$.

This equation also has two conservation laws, the mass and the energy:
\begin{gather*}
M(u(t))=\int u(t)^2=M_0,\\
E^{\gamma}(u(t))=\frac{1}{2}\int u_x(t)^2-\frac{1}{6}\int u(t)^6+\frac{\gamma}{q+1}\int |u(t)|^{q+1}=E^{\gamma}_0.
\end{gather*}
 
From the local wellposedness result obtained in \cite{KPV} and the two conservation laws above, we know that the solution of \eqref{CPG} is always global in time and bounded in $H^1$, and for all $t\in[0,+\infty)$, we have
$$\|u_x(t)\|^2_{L^2}\lesssim |E_0^{\gamma}|+\gamma^{-\frac{4}{q-5}}M_0<+\infty.$$

This equation does not have a standard scaling rule, but has the following pseudo-scaling rule: for all $\lambda_0>0$, if $u(t,x)$ is a solution to \eqref{CPG}, then
\begin{equation}\
u_{\lambda_0}(t,x)=\lambda_0^{-\frac{1}{2}}u(\lambda_0^{-3}t,\lambda_0^{-1}x),
\end{equation}
is a solution to
\begin{equation*}\begin{cases}
\partial_t v +(v_{xx}+v^5-\lambda_0^{-m}\gamma v|v|^{q-1})_x=0, \quad (t,x)\in[0,\lambda_0^{-3}T)\times\mathbb{R},\\
v(0,x)=\lambda_0^{-\frac{1}{2}}u_0(\lambda_0^{-1}x)\in H^1(\mathbb{R}),
\end{cases}
\end{equation*}
with 
\begin{equation}\label{101}
m=\frac{q-5}{2}>0.
\end{equation}
The pseudo-scaling rule leaves the $L^2$ norm of the initial data invariant. 

There also exist soliton solutions to \eqref{CPG}, given by
$$u(t,x)=\lambda_0^{-\frac{1}{2}}\mathcal{Q}_{\lambda_0^{-m}\gamma}\big(\lambda_0^{-1}(x-x_0)-\lambda_0^{-3}(t-t_0)\big).$$
for all $\lambda_0>0$, $t_0\in\mathbb{R}$, $x_0\in\mathbb{R}$ with $\lambda_0^{-m}\gamma\ll1$. Here for $0\leq\omega<\omega^*\ll1$, $\mathcal{Q}_{\omega}$ is the unique radial nonnegative solution with exponential decay to the following ODE%
\footnote{The existence of such $\mathcal{Q}_{\omega}$ was first proved in Section 6 of \cite{BL}. An alternative proof was given in Section 2.1 of \cite{L3}}%
:
$$\mathcal{Q}_{\omega}''-\mathcal{Q}_{\omega}+\mathcal{Q}_{\omega}^5-\omega \mathcal{Q}_{\omega}|\mathcal{Q}_{\omega}|^{q-1}=0.$$

In \cite{L3}, Lan obtained a similar classification result for the asymptotic dynamics of \eqref{CPG} near the ground state $\mathcal{Q}_{\gamma}$.

More precisely, we fix a small universal constant $\omega^*>0$ (to ensure the existence of the ground state $\mathcal{Q}_{\omega}$), and then introduce the following $L^2$ tube around $\mathcal{Q}_{\gamma}$:
$$\mathcal{T}_{\alpha^*,\gamma}=\bigg\{u_0\in H^1\Big|\inf_{\lambda_0>0,\lambda_0^{-m}\gamma<\omega^*,x_0\in\mathbb{R}}\bigg\|u_0-\frac{1}{\lambda_0^{\frac{1}{2}}}\mathcal{Q}_{\lambda_0^{-m}\gamma}\bigg(\frac{x-x_0}{\lambda_0}\bigg)\bigg\|_{L^2}<\alpha^*\bigg\}.$$
Then we have:
\begin{theorem}[Dynamics in $\mathcal{A}_{\alpha_0}$]\label{PT}
For all $q>5$, there exists a constant $0<\alpha^*(q)\ll1$, such that if $0<\gamma\ll\alpha_0\ll\alpha^*<\alpha^*(q)$, then for all $u_0\in\mathcal{A}_{\alpha_0}$, the corresponding solution $u(t)$ to \eqref{CPG} has one and only one of the following behaviors:\\
	{\rm\bf(Soliton):} For all $t\in[0,+\infty)$, $u(t)\in\mathcal{T}_{\alpha^*,\gamma}$. Moreover,  there exist a constant $\lambda_{\infty}\in(0,+\infty)$ and a $C^1$ function $x(t)$ such that
	\begin{gather}
	\lambda_{\infty}^{\frac{1}{2}}u(t,\lambda_{\infty}\cdot+x(t))\rightarrow \mathcal{Q}_{\lambda_{\infty}^{-m}\gamma} \text{ in } H^1_{{\rm loc}}, \text{ as $t\rightarrow +\infty$};\label{12}\\
	x(t)\sim\frac{t}{\lambda_{\infty}^2},\quad\text{as $t\rightarrow +\infty$}.
	\end{gather}
	{\rm\bf(Blow down):} For all $t\in[0,+\infty)$, $u(t)\in\mathcal{T}_{\alpha^*,\gamma}$. Moreover, there exist two $C^1$ functions $\lambda(t)$ and $x(t)$, such that
	\begin{gather}
	\lambda^{\frac{1}{2}}(t)u(t,\lambda(t)\cdot+x(t))\rightarrow Q\text{ in } H^1_{{\rm loc}}, \text{ as $t\rightarrow +\infty$};\\
	\lambda(t)\sim t^{\frac{2}{q+1}},\quad x(t)\sim t^{\frac{q-3}{q+1}},\quad\text{as $t\rightarrow +\infty$},
	\end{gather}
	{\rm\bf(Exit):} There exists a $0<t^*_{\gamma}<+\infty$ such that $u(t^*_{\gamma})\notin\mathcal{T}_{\alpha^*,\gamma}$.
	
	There exist solutions associated to each regime. Moreover, the regime (Soliton) and (Exit) are stable under small perturbation in $\mathcal{A}_{\alpha_0}.$ 
\end{theorem}

\begin{theorem}[Limiting case as $\gamma\rightarrow0$]\label{ST}
Let us fix a nonlinearity $q>5$, and choose $0<\alpha_0\ll\alpha^*<\alpha^*(q)$ as in Theorem \ref{PT}. For all $u_0\in\mathcal{A}_{\alpha_0}$, let $u(t)$ be the corresponding solution of \eqref{CP}, and $u_{\gamma}(t)$ be the corresponding solution of \eqref{CPG}. Then we have:
\begin{enumerate}
	\item If $u(t)$ is in the (Blow up) regime defined in Theorem \ref{T1}, then there exists $0<\gamma(u_0,\alpha_0,\alpha^*,q)\ll\alpha_0$ such that if $0<\gamma<\gamma(u_0,\alpha_0,\alpha^*,q)$, then $u_{\gamma}(t)$ is in the (Soliton) regime defined in Theorem \ref{PT}. Moreover, there exist constants $d_i=d_i(u_0,q)>0$, $i=1,2$, such that
	\begin{equation}
	\label{13}
	d_1\gamma^{\frac{2}{q-1}}\leq\lambda_{\infty}\leq d_2\gamma^{\frac{2}{q-1}},
	\end{equation}
	where $\lambda_{\infty}$ is the constant defined in \eqref{12}.
	\item If $u(t)$ is in the (Exit) regime defined in Theorem \ref{T1}, then there exists $0<\gamma(u_0,\alpha_0,\alpha^*,q)\ll\alpha_0$ such that if $0<\gamma<\gamma(u_0,\alpha_0,\alpha^*,q)$, then $u_{\gamma}(t)$ is in the (Exit) regime defined in Theorem \ref{PT}.
\end{enumerate} 
\end{theorem}

\begin{remark}
Theorem \ref{PT} shows that in the saturated setting there may be some different behavior (the \emph{blow down} behavior), which does not seem to happen in the unperturbed cases for solution with initial data in $\mathcal{A}_{\alpha_0}$. Examples for solution with a blow down behavior was also found by Donninger, Krieger \cite{DK} for energy critical wave equations. There are also examples of blow down behavior for $L^2$ critical NLS, where the blow down behavior can be obtained as the pseudo-conformal transformation of the log-log blow-up solutions.
\end{remark}

\subsection{Main result} 
The main purpose of this paper is to construct a natural continuation after the blow-up time for the $H^1$ blow-up solutions of \eqref{CP}.
This type of problems arising in physics has attracted a considerable attention in past few years but it is still poorly understood even at a formal level. 

One approach is to consider a sequence of globally defined approximate solutions $\{u_{\delta}(t)\}_{\delta>0}$ such that $u_{\delta}(t)$ converges (as $\delta\rightarrow0$) to the blow up solution $u(t)$ for all $t<T$, where $T<+\infty$ is the blow-up time. Then we expect that for $t>T$, the limit also exists and satisfies the original equation in some sense. And if this holds, the limiting function can be viewed as a natural extension of the blow-up solution $u(t)$ after the blow-up time $T$.

Examples of this approach were achieved in \cite{M2,M6,MRS3} for the focusing $L^2$-critical nonlinear Schr\"odinger equation:
\begin{equation}
\label{NLS}\tag{NLS}
\begin{cases}
i\partial_tu+\Delta u+|u|^{\frac{4}{d}}u=0,\;(t,x)\in\mathbb{R}\times\mathbb{R}^d,\\
u(0)=u_0\in H^1(\mathbb{R}^d),
\end{cases}
\end{equation}
where different ways to construct the approximation sequence $\{u_{\delta}(t)\}_{\delta>0}$ are introduced. In \cite{M6}, Merle constructed $\{u_{\delta}(t)\}_{\delta>0}$ as solutions to
$$\begin{cases}
i\partial_tu+\Delta u+|u|^{\frac{4}{d}-\delta}u=0,\;(t,x)\in\mathbb{R}\times\mathbb{R}^d,\\
u(0)=u_0\in H^1(\mathbb{R}^d).
\end{cases}$$
In \cite{MRS3}, Merle-Rapha\"el-Szeftel constructed $\{u_{\delta}(t)\}_{\delta>0}$ as global solutions to \eqref{NLS} with initial $u_{0,\delta}\in H^1$ such that $\lim_{\delta\rightarrow0}u_{0,\delta}=u_0$ in $H^1$. While in \cite{M2}, Merle constructed $\{u_{\delta}(t)\}_{\delta>0}$ as solutions to the $L^2$-critical NLS with a saturated perturbation, $i.e.$
\begin{equation}\label{11}
\begin{cases}
i\partial_tu+\Delta u+|u|^{\frac{4}{d}}u-\delta|u|^{q-1}u=0,\;(t,x)\in\mathbb{R}\times\mathbb{R}^d,\\
u(0)=u_0\in H^1(\mathbb{R}^d),
\end{cases}
\end{equation}
with $1+d/4<q<1+4/(d-2)$. On the other hand, the saturated perturbation like \eqref{11} is also considered as a correction to the NLS equations with pure power nonlinearities. See detailed discussion in \cite{GA,L3,0MRS} and the references therein.

In this paper, we follow similar arguments as in \cite{M2}, $i.e.$ consider the approximate sequence $\{u_{\gamma}(t)\}_{\gamma>0}$ as solutions to the saturated problem \eqref{CPG} with $\gamma>0$. 

For this approximate sequence, we may ask the following questions:
\begin{itemize}
\item {\bf(Compactness)} Is there a compact behavior for $u_{\gamma}(t)$ as $\gamma\rightarrow 0$, or equivalently are there a subsequence $\gamma_n\rightarrow 0$, and a function $u^{\infty}(t)$ such that $u_{\gamma_n}(t)\rightarrow u^{\infty}(t)$, as $n\rightarrow +\infty$? And in which sense does this limiting function $u^{\infty}(t)$ satisfy the unperturbed gKdV equation \eqref{CP}?
\item {\bf(Uniqueness)} Is the limiting function $u^{\infty}(t)$ unique or equivalently does $u^{\infty}(t)=\lim_{\gamma\rightarrow0}u_{\gamma}(t)$ hold for all $t>T$? And if this does not hold, what information is lost?
\item {\bf (Stability)} Is the blow-up phenomenon stable or equivalently do we have
$$\limsup_{\gamma\rightarrow0}\|u_{\gamma}(t)\|_{H^1}=+\infty$$
for all $t\geq T$?
\item {\bf (Continuity)} Does the limiting function (if it exists) depend continuously on the initial data?
\end{itemize}

Thanks to the work of \cite{L3}, we may give a precise answer to the above questions. Indeed, we have:
\begin{theorem}\label{MT}
Let $0<\alpha_0\ll1$ be the universal constant introduced in Theorem \ref{T1} and \ref{PT}, and $u_0\in \mathcal{A}_{\alpha_0}$ such that the corresponding solution $u(t)$ to \eqref{CP} belongs to the (Blow up) regime introduced in Theorem \ref{T1}.  Let $T<+\infty$ be the corresponding blow-up time. Now, for $q>5$ and $\gamma>0$ small enough, we denote by $u_{\gamma}(t)$, the solution to \eqref{CPG} with initial data $u_{\gamma}(0)=u_0$. We also denote by 
$$u_{\rm ext}(t)=
\begin{cases}
u(t),\;\text{if }t\in[0,T),\\
v(t),\;\text{if }t\in[T,+\infty),
\end{cases}
$$
where $v(t)$ is the unique global solution to \eqref{CP} with%
\footnote{Recall that $u^*\in H^1$ is the limiting profile introduced in Theorem \ref{T1}. }
\begin{equation}
v(T)=u^*\in H^1.
\end{equation}

Then, we have:
\begin{enumerate}
\item For all $T_0>T$, $R>0$, we have $u_{\rm ext}(t)\in\mathcal{C}([0,T_0], L^2(|x|<R))$, and
\begin{equation}
\label{15}
u_{\gamma}(t)\rightarrow u_{\rm ext}(t)\text{ in } \mathcal{C}\big([0,T_0], L^2(|x|<R)\big), \;\;\text{as }\gamma\rightarrow0,
\end{equation}
\item The limiting function $u_{\rm ext}(t)$ is a weak solution to \eqref{CP} in the following sense: for all $z(t,x)\in C^{\infty}_0(\mathbb{R}\times\mathbb{R})$, we have for all $t\in[0,+\infty)$:
\begin{align}
\label{16}
&\int_{\mathbb{R}}u_{\rm ext}(t,x)z(t,x)\,dx-\int_{\mathbb{R}}u_0(x)z(0,x)\,dx\nonumber\\
&=\int_0^t\bigg\{\int_{\mathbb{R}}u_{\rm ext}(s,x)\partial_tz(s,x)\,dx+\int_{\mathbb{R}}u_{\rm ext}(s,x)\partial_x^3z(s,x)\,dx\nonumber\\
&\qquad+\int_{\mathbb{R}}u^5_{\rm ext}(s,x)\partial_xz(s,x)\,dx\bigg\}\,ds.
\end{align}
\end{enumerate}
\end{theorem}

\noindent {\it Comments on  Theorem \ref{MT}:}

{\it 1. Global existence for $v(t)$.} From the arguments in \cite{MMR1}, we have $\|u^*\|_{H^1}\ll1$, which together with Theorem 2.8 in \cite{KPV} implies the global existence of $v(t)$ immediately.

{\it 2. Continuation after blow-up time for $L^2$-critical gKdV.} Theorem \ref{MT} shows that $\lim_{\gamma\rightarrow0}u_{\gamma}(t)$ exists in $\mathcal{C}([0,T_0], L^2(|x|<R))$, and the limiting function $u_{\rm ext}(t)$ satisfies \eqref{CP} in the weak sense, hence can be viewed as a natural extension of the blow-up solution $u(t)$. Moreover, on may easily check that the limiting function $u_{\rm ext}(t)$ depends continuously on the initial data in the stable blow-up regime. 

{\it 3. Regular behavior for the approximate sequence.} There is no singular behavior for $u_{\gamma}(t)$ as $\gamma\rightarrow0$. More precisely, the limiting function $u_{\rm ext}(t)$ is unique and the blow-up phenomenon is stable $i.e.$ for all $t\geq T$, we have
$$\|u_{\gamma}(t)\|_{H^1}\rightarrow+\infty,\; \text{as }\gamma\rightarrow0.$$
We mention here that these properties do not always hold true. For example, from \cite{M6,MRS3}, in the Schr\"odinger case, the limiting function for a special choice of approximate sequence $\{u_{\varepsilon}(t)\}_{\varepsilon>0}$ may not be unique. We have a loss of information on the phase in this case, see also in \cite{DFS,FK} for more detailed discussion. On the other hand, the blow-up phenomenon is unstable for $t>T$. More precisely, for all $t>T$,
$$\limsup_{\varepsilon\rightarrow0}\|u_{\varepsilon}(t)\|_{H^1}<+\infty.$$ 

{\it 4. On the exotic blow-up regime.} We expect to construct a similar extension for blow-up solutions to \eqref{CP} in the unstable regime (for example the solutions constructed in \cite{MMR2, MMR3}.) And due to the instability, we may expect some chaotic behavior for the approximate sequence $\{u_{\gamma}(t)\}_{\gamma>0}$ as $\gamma\rightarrow0$ (nonuniqueness of the limiting function, instability of the blow-up phenomenon $etc.$).

{\it 5. The supercritical case.} In \cite{L1}, Lan proved the existence and stability of self-similar blow-up solutions for slightly $L^2$-critical gKdV equations. Similar results can also be expected. But due to the supercritical structure, we know little about the asymptotic dynamics for the saturated problem in this case. Hence it is hard to apply the argument in this paper to the supercritical case. On the other hand, for the self-similar blow-up solutions constructed in \cite{L1}, the singularity concentrates on a finite point. This is different from the critical case, where the singularity goes to $+\infty$, as $t$ converges to the blow-up time. This fact may result in some irregular behavior for the continuation solution (for example, loss of some information or instability of blow-up phenomenon). But it is completely open. 

\subsection{Notation}
For $0\leq\omega<\omega^*\ll1$, we let $\mathcal{Q}_{\omega}$ be the unique nonnegative radial solution with exponential decay to the following ODE:
\begin{equation}\label{CME0}
\mathcal{Q}_{\omega}''-\mathcal{Q}_{\omega}+\mathcal{Q}_{\omega}^5-\omega \mathcal{Q}_{\omega}|\mathcal{Q}_{\omega}|^{q-1}=0.
\end{equation}
For simplicity, we denote by $Q=\mathcal{Q}_0$. Recall that we have:
$$Q(x)=\bigg(\frac{3}{\cosh^2(2x)}\bigg)^{\frac{1}{4}}.$$

We also introduce the linearized operator at $\mathcal{Q}_{\omega}$:
$$L_{\omega}f=-f''+f-5\mathcal{Q}_{\omega}^4f+q\omega|\mathcal{Q}_{\omega}|^{q-1}f.$$
Similarly, we denote by $L=L_0$.

Next, we introduce the scaling operator:
$$\Lambda f=\frac{1}{2}f+yf'.$$

Then, for a given small constant $\alpha>0$, we denote by $\delta(\alpha)$ a generic small constant with
$$\lim_{\alpha\rightarrow0}\delta(\alpha)=0.$$

Finally, we denote the $L^2$ scalar product by
$$(f,g)=\int f(x)g(x)\,dx.$$

\subsection*{Acknowledgment} The author would like to thank his Ph.D. supervisors F. Merle and T. Duyckaerts for introducing this problem to him and providing a lot of guidance.

\section{Overview on the asymptotic dynamics for perturbed and unperturbed gKdV equations}\label{S2}
In this section we collect a number of results which can be explicitly found in \cite{L3,MMR1}  and which we will use in the proof of Theorem \ref{MT}.

\subsection{The nonlinear profile}\label{S21}
Denote by $\mathcal{Y}$ the set of smooth function $f$ such that for all $k\in\mathbb{N}$, there exist $r_k>0$, $C_k>0$, with
$$|\partial_{y}^kf(y)|\leq C_k(1+|y|)^{r_k}e^{-|y|}.$$

We recall the construction of the nonlinear profiles%
\footnote{We mention here that we use a notation different from \cite{MMR1} to avoid misunderstanding.}
 $V_b$ for \eqref{CP} and $Q_{b,\omega}$ for \eqref{CPG}.

\begin{lemma}[{Nonlocalized profile, Proposition 2.4 in \cite{L3}, Proposition 2.2 in \cite{MMR1}}]
For all $0\leq\omega\ll1$, there exist a smooth function $P_{\omega}$ with $\partial_yP_{\omega}\in\mathcal{Y}$, such that:
\begin{gather}
(L_{\omega}P_{\omega})'=\Lambda \mathcal{Q}_{\omega},\quad \lim_{y\rightarrow-\infty}P_{\omega}(y)=\frac{1}{2}\int \mathcal{Q}_{\omega},\label{21}\\
(P_{\omega},\mathcal{Q}'_{\omega})=0,\quad  (P_{\omega},\mathcal{Q}_{\omega})=\frac{1}{16}\bigg(\int Q\bigg)^2+O(\omega).\label{22}
\end{gather}

Moreover there exist constants $C_0,C_1,\ldots$, independent of $\omega$, such that
\begin{gather}
|P_{\omega}(y)|+\bigg|\frac{\partial P_{\omega}}{\partial\omega}(y)\bigg|\leq C_0e^{-\frac{y}{2}},\text{ \rm for all }y>0,\\
|P_{\omega}(y)|+\bigg|\frac{\partial P_{\omega}}{\partial\omega}(y)\bigg|\leq C_0,\text{ \rm for all }y\in\mathbb{R},\\
|\partial_{y}^kP_{\omega}(y)|\leq C_ke^{-\frac{|y|}{2}}\text{ \rm for all $k\in\mathbb{N}_+$, $y\in\mathbb{R}$}.
\end{gather}
\end{lemma}

Now for $|b|\ll1$, $0\leq\omega\ll1$, we let 
\begin{align}
&Q_{b,\omega}(y)=\mathcal{Q}_{\omega}+b\chi(|b|^{\beta}y)P_{\omega}(y),\label{203}\\
&V_b(y)=Q_{b,0}(y),
\end{align}
where $\beta=\frac{3}{4}$.

Then we have the following properties of these two localized profiles:

\begin{lemma}[{Lemma 2.5 in \cite{L3}, Lemma 2.4 in \cite{MMR1}}]
For $|b|\ll1$, $0\leq\omega<\ll1$, there holds:
\begin{enumerate}
	\item Estimates on $Q_b$: For all $y\in\mathbb{R}$, $k\in\mathbb{N}$,
	\begin{align}
	&|Q_{b,\omega}(y)|\lesssim e^{-|y|}+|b|\big(\mathbf{1}_{[-2,0]}(|b|^{\beta}y)+e^{-\frac{|y|}{2}}\big),\label{201}\\
	&|\partial_{y}^kQ_{b,\omega}(y)|\lesssim e^{-|y|}+|b|e^{-\frac{|y|}{2}}+|b|^{1+k\beta}\mathbf{1}_{[-2,-1]}(|b|^{\beta}y),\label{202}
	\end{align}
	where $\mathbf{1}_{I}$ denotes the characteristic function of the interval $I$.
	\item Equation of $Q_{b,w}$: Let
	\begin{equation}\label{CAP}
	-\Psi_{b,\omega}=b\Lambda Q_{b,\omega}+\big(Q_{b,\omega}''-Q_{b,\omega}+Q_{b,\omega}^5-\omega Q_{b,\omega}|Q_{b,\omega}|^{q-1}\big)'.
	\end{equation}
	Then, for all $y\in\mathbb{R}$,
	\begin{align}
	-\Psi_{b,\omega}=&b^2\big((10\mathcal{Q}_{\omega}^3P_{\omega}^2)_{y}+\Lambda P_{\omega}\big)-\frac{1}{2}b^2(1-\chi_{b})P_{\omega}\nonumber\\
	&+O\Big(|b|^{1+\beta}\mathbf{1}_{[-2,-1]}(|b|^{\beta}y)+b^2(\omega+|b|)e^{-\frac{|y|}{2}}\Big).\label{23}
	\end{align}
	Moreover, we have
	\begin{equation}\label{24}
	|\partial_{y}\Psi_{b,\omega}(y)|\lesssim |b|^{1+2\beta}\mathbf{1}_{[-2,-1]}(|b|^{\beta}y)+b^2e^{-\frac{|y|}{2}}.
	\end{equation}
	\item Mass and energy properties of $Q_{b,\omega}$:
	\begin{gather}
	\Bigg|\int Q_{b,\omega}^2-\bigg(\int \mathcal{Q}_{\omega}^2+2b\int P_{\omega}\mathcal{Q}_{\omega}\bigg)\Bigg|\lesssim |b|^{2-\beta},\label{25}\\
	|E(Q_{b,\omega})|\lesssim |b|+\omega.\label{26}
	\end{gather}
\end{enumerate}
\end{lemma}

\begin{proof}
See Lemma 2.4 in \cite{MMR1} and Lemma 2.5 in \cite{L3}.
\end{proof}

\subsection{Geometrical decomposition of the flow and modulation estimates}\label{S22}
For simplicity, from this subsection, we fix a $u_0\in\mathcal{A}_{\alpha_0}$ such that the corresponding solution $u(t)$ to \eqref{CP} belongs to the (Blow-up) regime as in Theorem \ref{T1}. We denote by $T<+\infty$ its blow-up time. We also let $u_{\gamma}(t)$ be the corresponding solution to \eqref{CPG}, which belongs to the (Soliton) regime as indicated in Theorem \ref{PT} for $\gamma>0$ small enough.

\begin{lemma}[{Geometrical decomposition for $u(t)$, Lemma 2.5 in \cite{MMR1}}]
There exist three $C^1$ functions $(\lambda,x,b):\; [0,T)\rightarrow(0,+\infty)\times \mathbb{R}^2$, such that for all $t\in[0,T)$, there holds
\begin{equation}\label{GD}
u(t,x)=\frac{1}{\lambda^{\frac{1}{2}}(t)}\big[V_{b(t)}+\varepsilon(t)\big]\bigg(\frac{x-x(t)}{\lambda(t)}\bigg),
\end{equation}
with $\varepsilon(t)$ satisfying the following orthogonality conditions
\begin{equation}\label{OC}
(\varepsilon(t),Q)=(\varepsilon(t),\Lambda Q)=(\varepsilon(t),y\Lambda Q)=0,
\end{equation}
for all $t\in[0,T)$.
\end{lemma}

\begin{lemma}[{Geometrical decomposition for $u_{\gamma}(t)$, Lemma 2.6 in \cite{L3}}]
For $q>5$ and $\gamma>0$ small enough, there exist $C^1$ functions $(\lambda_{\gamma},x_{\gamma},b_{\gamma}):\; [0,+\infty)\rightarrow(0,+\infty)\times \mathbb{R}^2$, such that for all $t\in[0,+\infty)$, there holds
\begin{equation}\label{GDG}
u_{\gamma}(t,x)=\frac{1}{\lambda_{\gamma}^{\frac{1}{2}}(t)}\big[Q_{b_{\gamma}(t),\omega_{\gamma}(t)}+\varepsilon_{\gamma}(t)\big]\bigg(\frac{x-x_{\gamma}(t)}{\lambda_{\gamma}(t)}\bigg),
\end{equation}
where%
\footnote{Recall $m=(q-5)/2$ is defined in \eqref{101}.} 
$$\omega_{\gamma}(t)=\frac{\gamma}{\lambda^m_{\gamma}(t)}.$$
And $\varepsilon_{\gamma}(t)$ satisfies the following orthogonality conditions
\begin{equation}\label{OCG}
(\varepsilon_{\gamma}(t),\mathcal{Q}_{\omega(t)})=(\varepsilon_{\gamma}(t),\Lambda \mathcal{Q}_{\omega(t)})=(\varepsilon_{\gamma}(t),y\Lambda \mathcal{Q}_{\omega(t)})=0,
\end{equation}
for all $t\in[0,+\infty)$.

Moreover, for all $t\in[0,T)$, we have%
\footnote{See Lemma 2.6 and (5.17) in \cite{L3}}
\begin{equation}\label{239}
\big(\lambda_{\gamma}(t),b_{\gamma}(t),x_{\gamma}(t),\varepsilon_{\gamma}(t)\big)\xrightarrow{\mathbb{R}^3\times H^1}\big(\lambda(t),b(t),x(t),\varepsilon(t)\big),
\end{equation}
as $\gamma\rightarrow0$.
\end{lemma}

\begin{remark}
Generally, for $u_0\in\mathcal{A}_{\alpha_0}$, the geometrical decomposition \eqref{GD} and \eqref{GDG} may not hold true for all $t\in[0,T_{u_0})$, where $T_{u_0}$ is the maximal lifespan. But  for $u_0\in\mathcal{A}_{\alpha_0}$, such that $u(t)$ is in the (Blow-up) regime and $u_{\gamma}(t)$ is in the (Soliton) regime, this fact holds true. This follows from a bootstrap argument, which is the heart of the proof of Theorem \ref{T1} and Theorem \ref{PT}. For simplicity, we ignore the bootstrap argument and focus on such $u_0$ only.
\end{remark}

For $(\lambda_{\gamma},x_{\gamma},b_{\gamma})$, we have:
\begin{proposition}[{Modulation estimates for $u_{\gamma}(t)$, Proposition 2.9 in \cite{L3}}]\label{P2}
We let
$$s(t)=\int_{0}^t\frac{1}{\lambda_{\gamma}^3(\tau)}\,d\tau,\quad y=\frac{x-x_{\gamma}(t)}{\lambda_{\gamma}(t)}.$$
Then we have:
\begin{enumerate}
	\item {\rm(Equation of $\varepsilon_{\gamma}$)}: For all $s\in[0,+\infty)$,
	\begin{align}\label{213}
	\partial_s\varepsilon_{\gamma}=&(L_{\omega_{\gamma}}\varepsilon_{\gamma})_{y}-b_{\gamma}\Lambda\varepsilon_{\gamma}+\bigg(\frac{\partial_s\lambda_{\gamma}}{\lambda_{\gamma}}+b_{\gamma}\bigg)(\Lambda Q_{b_{\gamma},\omega_{\gamma}}+\Lambda\varepsilon_{\gamma})\nonumber\\
    &+\bigg(\frac{\partial_sx_{\gamma}}{\lambda_{\gamma}}-1\bigg)(Q_{b_{\gamma},\omega_{\gamma}}+\varepsilon_{\gamma})_{y}-\partial_sb_{\gamma}\frac{\partial Q_{b_{\gamma},\omega_{\gamma}}}{\partial b_{\gamma}}-\partial_s\omega_{\gamma}\frac{\partial Q_{b_{\gamma},\omega_{\gamma}}}{\partial \omega_{\gamma}}\nonumber\\
    &+\Psi_{b_{\gamma},\omega_{\gamma}}-(R_{b_{\gamma}}^{\gamma}(\varepsilon_{\gamma}))_{y}-(R^{\gamma}_{\rm NL}(\varepsilon_{\gamma}))_{y},
	\end{align}
	where
	\begin{align}
	&\Psi_{b_{\gamma},\omega_{\gamma}}=-b_{\gamma}\Lambda Q_{b_{\gamma},\omega_{\gamma}}-\big(Q_{b_{\gamma},\omega_{\gamma}}''-Q_{b_{\gamma},\omega_{\gamma}}+Q_{b_{\gamma},\omega_{\gamma}}^5-\omega_{\gamma} Q_{b_{\gamma},\omega_{\gamma}}|Q_{b_{\gamma},\omega_{\gamma}}|^{q-1}\big)',\\
	&R_{b_{\gamma}}^{\gamma}(\varepsilon_{\gamma})=5(Q_{b_{\gamma},\omega_{\gamma}}^4-\mathcal{Q}_{\omega_{\gamma}}^4)\varepsilon_{\gamma}-q\omega_{\gamma}(|Q_{b_{\gamma},\omega_{\gamma}}|^{q-1}-|\mathcal{Q}_{\omega_{\gamma}}|^{q-1})\varepsilon_{\gamma},\\
	&R^{\gamma}_{\rm NL}(\varepsilon_{\gamma})=(\varepsilon_{\gamma}+Q_{b_{\gamma},\omega_{\gamma}})^5-5Q_{b_{\gamma},\omega_{\gamma}}^{4}\varepsilon_{\gamma}-Q_{b_{\gamma},\omega_{\gamma}}^5\nonumber\\
    &\quad-\omega_{\gamma}\big[(\varepsilon_{\gamma}+Q_{b_{\gamma},\omega_{\gamma}})|\varepsilon_{\gamma}+Q_{b_{\gamma},\omega_{\gamma}}|^{q-1}-Q_{b_{\gamma},\omega_{\gamma}}|Q_{b_{\gamma},\omega_{\gamma}}|^{q-1}-q\varepsilon_{\gamma}|Q_{b_{\gamma},\omega_{\gamma}}|^{q-1}\big].
	\end{align}
	\item {\rm (Estimates induced by the conservation laws)}. For $s\in[0,+\infty)$, there holds:
	\begin{gather}
	\|\varepsilon_{\gamma}\|_{L^2}\lesssim|b_{\gamma}|^{\frac{1}{4}}+\omega_{\gamma}^{\frac{1}{2}}+\bigg|\int u^2_0-\int Q^2\bigg|^{\frac{1}{2}},	\label{MC}\\
	\frac{\|\partial_y\varepsilon_{\gamma}\|_{L^2}^2}{\lambda_{\gamma}^2}\lesssim\frac{1}{\lambda_{\gamma}^2}\bigg(\omega_{\gamma}+|b_{\gamma}|+\int\varepsilon_{\gamma}^2e^{-\frac{|y|}{10}}\bigg)+\gamma\frac{\|\partial_y\varepsilon_{\gamma}\|_{L^2}^{m+2}}{\lambda_{\gamma}^{m+2}}+|E^{\gamma}_0|.\label{EC}
	\end{gather}
	\item {\rm ($H^1$ modulation equation)}. For all $s\in[0,+\infty)$,
	\begin{align}
	&\bigg|\frac{\partial_s\lambda_{\gamma}}{\lambda_{\gamma}}+b_{\gamma}\bigg|+\bigg|\frac{\partial_sx_{\gamma}}{\lambda_{\gamma}}-1\bigg|\lesssim \bigg(\int\varepsilon_{\gamma}^2e^{-\frac{|y|}{10}}\bigg)^{\frac{1}{2}}+|b_{\gamma}|(\omega_{\gamma}+|b_{\gamma}|),\label{MS1} \\
	&|\partial_sb_{\gamma}|+|\partial_s\omega_{\gamma}|\lesssim(\omega_{\gamma}+|b_{\gamma}|)\Bigg[\bigg(\int\varepsilon_{\gamma}^2e^{-\frac{|y|}{10}}\bigg)^{\frac{1}{2}}+|b_{\gamma}|\Bigg]+\int\varepsilon_{\gamma}^2e^{-\frac{|y|}{10}}.\label{MS2}
	\end{align}
	\item {\rm ($L^1$ control on the right)}. Let
    \begin{align}
    &\rho_1(y)=\frac{4}{(\int Q)^2}\int_{-\infty}^{y}\Lambda Q,\label{27}\\
    &\rho_2=\frac{16}{(\int Q)^2}\bigg(\frac{(\Lambda P,Q)}{\|\Lambda Q\|^2_{L^2}}\Lambda Q+P-\frac{1}{2}\int Q\bigg)-8\rho_1,\label{28}\\
    &\rho=4\rho_1+\rho_2,\label{29}
    \end{align}
    and
    \begin{equation}\label{214}
    J_1^{\gamma}(s)=(\varepsilon_{\gamma}(s),\rho_1),\;J_2^{\gamma}(s)=(\varepsilon_{\gamma}(s),\rho_2),\; J^{\gamma}(s)=(\varepsilon_{\gamma}(s),\rho),
    \end{equation}
    where $\rho_1$, $\rho_2$, $\rho$ were defined in \eqref{27}--\eqref{29}. Then we have:
	\begin{enumerate}
		\item {\rm (Law of $\lambda_{\gamma}$):} for all $s\in[0,+\infty)$,
		\begin{align}
		&\bigg|\frac{\partial_s\lambda_{\gamma}}{\lambda_{\gamma}}+b_{\gamma}-2\bigg((J^{\gamma}_1)_s+\frac{1}{2}\frac{\partial_s\lambda_{\gamma}}{\lambda_{\gamma}}J^{\gamma}_1\bigg)\bigg|\nonumber\\
		&\lesssim (\omega_{\gamma}+|b_{\gamma}|)\Bigg[\bigg(\int\varepsilon_{\gamma}^2e^{-\frac{|y|}{10}}\bigg)^{\frac{1}{2}}+|b_{\gamma}|\Bigg]+\int\varepsilon_{\gamma}^2e^{-\frac{|y|}{10}}.\label{LOL}
		\end{align}
		\item {\rm (Law of $b_{\gamma}$):} for all $s\in[0,+\infty)$,
		\begin{align}
		&\bigg|\partial_sb_{\gamma}+2b_{\gamma}^2+\partial_s\omega_{\gamma}G'(\omega_{\gamma})+b_{\gamma}\bigg((J^{\gamma}_2)_s+\frac{1}{2}\frac{\partial_s\lambda_{\gamma}}{\lambda_{\gamma}}J^{\gamma}_2\bigg)\bigg|\nonumber\\
		&\lesssim \int \varepsilon_{\gamma}^2e^{-\frac{|y|}{10}}+(\omega_{\gamma}+|b_{\gamma}|)b_{\gamma}^2,\label{LOB}
		\end{align}
		where $G\in C^2$ with $G(0)=0$, $G'(0)=c_0>0$, for some universal constant $c_0$.
		\item {\rm (Law of $\frac{b_{\gamma}}{\lambda_{\gamma}^2}$):} for all $s\in[0,+\infty)$,
		\begin{align}
		&\bigg|\frac{d}{ds}\bigg(\frac{b_{\gamma}}{\lambda_{\gamma}^2}\bigg)+\frac{b_{\gamma}}{\lambda_{\gamma}^2}\bigg(J^{\gamma}_s+\frac{1}{2}\frac{\partial_s\lambda_{\gamma}}{\lambda}J^{\gamma}\bigg)+\frac{\partial_s\omega_{\gamma}G'(\omega_{\gamma})}{\lambda_{\gamma}^2}\bigg|\nonumber\\
		&\lesssim \frac{1}{\lambda_{\gamma}^2}\bigg(\int \varepsilon_{\gamma}^2e^{-\frac{|y|}{10}}+(\omega_{\gamma}+|b_{\gamma}|)b_{\gamma}^2\bigg).\label{LOBOLS}
		\end{align}
	\end{enumerate}
\end{enumerate}
\end{proposition}

\subsection{Monotonicity formula and estimate on the error term}\label{S23}
We now recall the the monotonicity formula introduced in \cite{MMR1}, which is the heart of the analysis in \cite{MMR1}. We mention here again, for simplicity, we ignore the bootstrap argument and focus only on the initial data whose corresponding solution $u(t)$ to \eqref{CP} belongs to the (Blow-up) regime as in Theorem \ref{T1}.

More precisely, we let $(\varphi_i)_{i=1,2},\psi\in C^{\infty}(\mathbb{R})$ be such that:
\begin{gather}
\varphi_i(y)=
\begin{cases}
e^y,\text{ for }y<-1,\\
1+y,\text{ for }-\frac{1}{2}<y<\frac{1}{2},\\
y^i,\text{ for }y>2,
\end{cases}
\quad \varphi'(y)>0,\text{ for all }y\in\mathbb{R},\\
\psi(y)=
\begin{cases}
e^{2y},\text{ for }y<-1,\\
1,\text{ for }y>-\frac{1}{2},
\end{cases}
\quad \psi'(y)\geq0,\text{ for all }y\in\mathbb{R}.
\end{gather}  
Let $B>100$ be a large universal constant. We then define the following weight functions:
\begin{equation}
\psi_B(y)=\psi\bigg(\frac{y}{B}\bigg),\quad \varphi_{i,B}(y)=\varphi_i\bigg(\frac{y}{B}\bigg),
\end{equation} 
and the following weighted $H^1$ norm of $\varepsilon$: for all $s\in[0,+\infty)$,
\begin{gather}
\mathcal{N}_i(s)=\int\bigg(\varepsilon_{y}^2(s,y)\psi_B(y)+\varepsilon^2(s,y)\varphi_{i,B}(y)\bigg)\,dy,\quad i=1,2,\\
\mathcal{N}_{i,\rm loc}(s)=\int \varepsilon^2(s,y)\varphi'_{i,B}(y)\,dy,\quad i=1,2.
\end{gather}

Similarly, for $u_{\gamma}(t)$, we define:
\begin{gather}
\mathcal{N}^{\gamma}_i(s)=\int\bigg(\big[\partial_y\varepsilon_{\gamma}(s,y)\big]^2\psi_B(y)+\varepsilon_{\gamma}^2(s,y)\varphi_{i,B}(y)\bigg)\,dy,\quad i=1,2,\\
\mathcal{N}^{\gamma}_{i,\rm loc}(s)=\int \varepsilon_{\gamma}^2(s,y)\varphi'_{i,B}(y)\,dy,\quad i=1,2.
\end{gather}

Then we have the following monotonicity formula for $u_{\gamma}(t)$:
\begin{proposition}[{Monotonicity formula for $u_{\gamma}(t)$, Proposition 3.1 in \cite{L3}}]\label{P4}
For $\gamma>0$ small enough, we define the Lyapounov functional for $(i,j)\in\{1,2\}^2$ as following:
\begin{align}
\label{217}
&\mathcal{F}^{\gamma}_{i,j}(s)=\nonumber\\
&\int\bigg((\partial_y\varepsilon_{\gamma})^2\psi_B+(1+\mathcal{J}^{\gamma}_{i,j})\varepsilon_{\gamma}^2\varphi_{i,B}-\frac{1}{3}\psi_B\big[(Q_{b_{\gamma},\omega_{\gamma}}+\varepsilon_{\gamma})^6-Q_{b_{\gamma},\omega_{\gamma}}^6-6\varepsilon_{\gamma} Q_{b_{\gamma},\omega_{\gamma}}^5\big]\nonumber\\
&\quad+\frac{2\omega_{\gamma}}{q+1}\big[|Q_{b_{\gamma},\omega_{\gamma}}+\varepsilon_{\gamma}|^{q+1}-|Q_{b_{\gamma},\omega_{\gamma}}|^{q+1}-(q+1)\varepsilon_{\gamma} Q_{b_{\gamma},\omega_{\gamma}}|Q_{b_{\gamma},\omega_{\gamma}}|^{q-1}\big]\psi_B\bigg),
\end{align}
with
\begin{equation}
\mathcal{J}^{\gamma}_{i,j}=(1-J^{\gamma}_1)^{-4(j-1)-2i}-1.
\end{equation}
Then the following estimates hold for all $s\in[0,+\infty)$:
\begin{enumerate}
	\item Scaling invariant Lyapounov control: for $i=1,2$,
	\begin{equation}
	\label{MF1}
	\frac{d\mathcal{F}^{\gamma}_{i,1}}{ds}+\mu\int\Big((\partial_y\varepsilon_{\gamma})^2+\varepsilon_{\gamma}^2\Big)\varphi'_{i,B}\lesssim_B b_{\gamma}^2(\omega_{\gamma}^2+b_{\gamma}^2),
	\end{equation}
    where $\mu>0$ is a universal constant.
	\item $H^1$ scaling Lyapounov control: for $i=1,2$,
	\begin{equation}
	\label{MF2}
	\frac{d}{ds}\bigg(\frac{\mathcal{F}^{\gamma}_{i,2}}{\lambda_{\gamma}^2}\bigg)+\frac{\mu}{\lambda_{\gamma}^2}\int\Big((\partial_y\varepsilon_{\gamma})^2+\varepsilon_{\gamma}^2\Big)\varphi'_{i,B}\lesssim_B \frac{b_{\gamma}^2(\omega_{\gamma}^2+b_{\gamma}^2)}{\lambda_{\gamma}^2}.
	\end{equation}
	\item Coercivity and pointwise bounds: there holds for all $(i,j)\in\{1,2\}^2$, and all $s\in[0,+\infty)$,
	\begin{gather}
	\mathcal{N}^{\gamma}_i\lesssim\mathcal{F}^{\gamma}_{i,j}\lesssim \mathcal{N}^{\gamma}_i,\label{COER}\\
	|J^{\gamma}_i|+|\mathcal{J}^{\gamma}_{i,j}|\lesssim(\mathcal{N}^{\gamma}_2)^{\frac{1}{2}}\label{218}.
	\end{gather}
\end{enumerate}
\end{proposition}

As a consequence of the modulation estimates introduced in Section \ref{S22} and the monotonicity formulas introduced above, we have the following control on the error term $\varepsilon_{\gamma}(t)$:

\begin{lemma}[{Control of the error term for $u_{\gamma}(t)$, Lemma 4.1 in \cite{L3}}]\label{L0}
We have the following:
\begin{enumerate}
	\item {\rm (Almost monotonicity of the localized $H^1$ norm: there exists a universal constant $K_0>1$, such that for $i=1,2$ and $0\leq s_1<s_2< +\infty$, there holds):}
	\begin{align}
	&\mathcal{N}^{\gamma}_{i}(s_2)+\int_{s_1}^{s_2}\bigg(\int\big[(\partial_y\varepsilon_{\gamma})^2(s)+\varepsilon_{\gamma}^2(s)\big]\varphi'_{i,B}\bigg)\,ds\nonumber\\
	&\leq K_0\bigg[\mathcal{N}^{\gamma}_i(s_1)+\sup_{s\in[s_1,s_2]}|b_{\gamma}(s)|^3+\sup_{s\in[s_1,s_2]}\omega_{\gamma}^3(s)\bigg],\label{219}
    \end{align}
    and
    \begin{align}
	&\frac{\mathcal{N}^{\gamma}_i(s_2)}{\lambda_{\gamma}^2(s_2)}+\int_{s_1}^{s_2}\frac{1}{\lambda_{\gamma}^2(s)}\bigg[\bigg(\int\big[(\partial_y\varepsilon_{\gamma})^2(s)+\varepsilon_{\gamma}^2(s)\big]\varphi'_{i,B}\bigg)+b_{\gamma}^2(s)\big(|b_{\gamma}(s)|+\omega_{\gamma}(s)\big)\bigg]\,ds\nonumber\\
	&\leq K_0\Bigg(\frac{\mathcal{N}^{\gamma}_i(s_1)}{\lambda_{\gamma}^2(s_1)}+\bigg[\frac{b_{\gamma}^2(s_1)+\omega_{\gamma}^2(s_1)}{\lambda_{\gamma}^2(s_1)}+\frac{b_{\gamma}^2(s_2)+\omega_{\gamma}^2(s_2)}{\lambda_{\gamma}^2(s_2)}\bigg]\Bigg).\label{220}
	\end{align}
	\item {\rm (Control of $b_{\gamma}$ and $\omega_{\gamma}$):} for all $0\leq s_1<s_2<+\infty$, there holds:
	\begin{equation}
	\label{221}
	\omega_{\gamma}(s_2)+\int_{s_1}^{s_2}b_{\gamma}^2(s)\,ds\lesssim\mathcal{N}^{\gamma}_1(s_1)+\omega_{\gamma}(s_1)+\sup_{s\in[s_1,s_2]}|b_{\gamma}(s)|,
	\end{equation}
	\item {\rm (Control of $\frac{b_{\gamma}}{\lambda_{\gamma}^2}$):} let $c_1=\frac{m}{m+2}G'(0)>0$, where $G$ is the $C^2$ function introduced in \eqref{LOB}. Then	there exists a universal constant $K_1>1$ such that for all $0\leq s_1<s_2< +\infty$, there holds:
	\begin{align}
	&\bigg|\frac{b_{\gamma}(s_2)+c_1\omega_{\gamma}(s_2)}{\lambda_{\gamma}^2(s_2)}-\frac{b_{\gamma}(s_1)+c_1\omega_{\gamma}(s_1)}{\lambda_{\gamma}^2(s_1)}\bigg|\nonumber\\
	&\leq K_1\bigg(\frac{\mathcal{N}^{\gamma}_1(s_1)}{\lambda_{\gamma}^2(s_1)}+\frac{b_{\gamma}^2(s_1)+\omega_{\gamma}^2(s_1)}{\lambda_{\gamma}^2(s_1)}+\frac{b_{\gamma}^2(s_2)+\omega_{\gamma}^2(s_2)}{\lambda_{\gamma}^2(s_2)}\bigg).\label{222}
	\end{align}
	\item {\rm (Refined control of $\lambda_{\gamma}$):} let $\lambda^{\gamma}_0(s)=\lambda_{\gamma}(s)(1-J^{\gamma}_1(s))^2$. Then there exists a universal constant $K_2>1$ such that for all $s\in[0,+\infty)$,
	\begin{equation}\label{223}
	\bigg|\frac{(\lambda^{\gamma}_0)_s}{\lambda^{\gamma}_0}+b_{\gamma}\bigg|\leq K_2\Big(\mathcal{N}^{\gamma}_1+(|b_{\gamma}|+\omega_{\gamma})\big[(\mathcal{N}^{\gamma}_2)^{\frac{1}{2}}+|b_{\gamma}|\big]\Big).
	\end{equation}
\end{enumerate}
\end{lemma}

\subsection{Asymptotic dynamics in the (Soliton) region for \eqref{CPG}.}\label{S24}
This subsection is devoted to introduce some basic properties for the solution $u(t)$ to \eqref{CP} in the (Blow up) region and the solution $u_{\gamma}(t)$ to \eqref{CPG} in the (Soliton) region.

We fix a $u_0\in\mathcal{A}_{\alpha_0}$, such that the corresponding solution $u(t)$ to \eqref{CP} belongs to the (Blow-up) regime. We denote by $T<+\infty$, the blow-up time. We also let $\gamma<\gamma(u_0,\alpha_0,q)$ small enough, such that the corresponding solution $u_{\gamma}(t)$ to \eqref{CPG} belongs to the (Soliton) regime%
\footnote{This is ensured by Theorem \ref{ST}.}%
.

Now, from Proposition 6.1 in \cite{MMR1}, we have:
\begin{equation}
\label{224}
\tilde{u}(t)\rightarrow u^*\;\text{in }L^2,\;\text{as }t\rightarrow T,
\end{equation}
where 
$$\tilde{u}(t,x)=\frac{1}{\lambda^{1/2}(t)}\varepsilon\bigg(t,\frac{x-x(t)}{\lambda(t)}\bigg).$$
Moreover, there exist a constant $\ell_0=\ell_0(u_0)>0$, such that%
\footnote{See (4.7) in \cite{MMR1}.}
\begin{equation}\label{238}
\lim_{t\rightarrow T}\frac{\lambda(t)}{T-t}=\ell_0>0,\quad\lim_{t\rightarrow T}\frac{b(t)}{(T-t)^2}=\ell^3_0,\quad\lim_{t\rightarrow T}(T-t)^2x(t)=\ell^{-2}_0.
\end{equation}

From (2.27)--(2.28) and (4.54) in \cite{MMR1}, we also have $u^*\in H^1$, satisfying 
\begin{equation}
\label{225}
\|u^*\|_{H^1}\lesssim \delta(\alpha_0)\ll1.
\end{equation}

Next, we let $v(t)$ be the solution to \eqref{CP} with 
$$v(T)=u^*.$$
It is easy to see from Theorem 2.8 in \cite{KPV} and \eqref{32} that $v(t)$ exists globally in time and scatters at both time directions, $i.e.$
$$\exists v^{\pm\infty}\in L^2,\text{ such that }\lim_{t\rightarrow\pm\infty}\|v(t)-e^{-t\partial_x^3}v^{\pm\infty}\|_{L^2}=0,$$
or equivalently
\begin{equation}
\label{226}
\|v\|_{L_x^5L_t^{10}(\mathbb{R})}\lesssim 1.
\end{equation}

From (4.43)--(4.45) in \cite{MMR1}, we know that there exists a $t_1^*<T$ such that for all $t\in[0,t_1^*]$ we have
\begin{gather}
\mathcal{N}_2(t)+\|\varepsilon(t)\|_{H^1}+|b(t)|+|1-\lambda(t)|\lesssim\delta(\alpha_0),\label{227}\\
\int_{y>0}y^{10}\varepsilon^2(t,y)\,dy\leq 5.\label{228}\\
b(t^*_1)\geq 2C^*\mathcal{N}_1(t^*_1),\label{229}
\end{gather}
where $C^*$ is defined as following%
\footnote{Recall that $K_0$, $K_1$ and $K_2$ are defined in Lemma \ref{L0}.}%
:
$$C^*=100(K_0K_2+K_1).$$
Then for $0<\gamma<\gamma(u_0,\alpha_0)$ small enough, we have for all $t\in[0,t_1^*]$ the solution $u_{\gamma}(t)$ satisfies%
\footnote{See (4.52)--(4.55) and (5.18) in \cite{L3}.}%
:
\begin{gather}
\mathcal{N}^{\gamma}_2(t)+\|\varepsilon_{\gamma}(t)\|_{H^1}+|b_{\gamma}(t)|+|1-\lambda_{\gamma}(t)|\lesssim\delta(\alpha_0),\label{230}\\
\int_{y>0}y^{10}\varepsilon_{\gamma}^2(t,y)\,dy\leq 5.\label{231}\\
b_{\gamma}(t^*_1)\geq C^*\mathcal{N}^{\gamma}_1(t^*_1),\label{232}
\end{gather}

Then from Section 4 of \cite{L3}, we know that there exists a $t^*_{2,\gamma}\in(t_1^*,+\infty)$ such that $b_{\gamma}(t_{2,\gamma}^*)=\frac{1}{100}\omega_{\gamma}(t_{2,\gamma}^*)$ and for all $t\in [t_1^*,t_{2,\gamma}^*]$, there holds
\begin{equation}
b_{\gamma}(t)\geq \frac{1}{100}\omega_{\gamma}(t).\label{233}
\end{equation}
We also have for all $t_1^*\leq t_1<t_2\leq t_{2,\gamma}^*$, there holds%
\footnote{See (4.69) in \cite{L3}.}
\begin{equation}
\label{234}
\lambda_{\gamma}(t_2)\leq 2\lambda_{\gamma}(t_1).
\end{equation}
And for all%
\footnote{See (4.71) in \cite{L3}.}
 $t\in [t_1^*,t_{2,\gamma}^*]$,
\begin{equation}
0<\frac{\mathcal{N}_1^{\gamma}(t)}{\lambda_{\gamma}^2(t)}+\frac{\omega_{\gamma}(t)}{\lambda_{\gamma}^2(t)}\lesssim\frac{b_{\gamma}(t)}{\lambda^2_{\gamma}(t)}\sim \ell^*,\label{235}
\end{equation}
where 
$$\ell^*=\frac{b(t_1^*)}{\lambda^2(t_1^*)}>0,$$
independent of $\gamma$.

While for $t\in[t^*_{2,\gamma},+\infty)$, we have%
\footnote{See (4.92)--(4.94) in \cite{L3}.}
\begin{gather}
\lambda_{\gamma}(t)\sim \bigg(\frac{\gamma}{\ell^*}\bigg)^{\frac{1}{m+2}},\quad\omega_{\gamma}(t)\sim \gamma^{\frac{2}{m+2}}(\ell^*)^{\frac{m}{m+2}},\label{236}\\
\frac{\mathcal{N}_1^{\gamma}(t)}{\lambda_{\gamma}^2(t)}+\bigg|\frac{b_{\gamma}(t)}{\lambda^2_{\gamma}(t)}\bigg|\lesssim\frac{\omega_{\gamma}(t)}{\lambda_{\gamma}^2(t)}\sim \ell^*.\label{237}
\end{gather}

Finally, for all $t\in[t^*_1,+\infty)$, we have%
\footnote{See (4.82), (4.93)--(4.95) in \cite{L3}.}%
:
\begin{equation}
\label{240}
\mathcal{N}_2^{\gamma}(t)+|b_{\gamma}(t)|+\omega_{\gamma}(t)\lesssim \delta(\alpha_0)\ll1.
\end{equation}

\section{Continuation after blow-up time}\label{S3}
In this section, we will give the proof of Theorem \ref{MT}, using the analysis tools introduced in Section \ref{S2}.

First, we denote by 
\begin{gather}
\label{31}
\tilde{u}_{\gamma}(t,x)=\frac{1}{\lambda_{\gamma}^{1/2}(t)}\varepsilon_{\gamma}\bigg(t,\frac{x-x_{\gamma}(t)}{\lambda_{\gamma}(t)}\bigg),\\ Q_S^{\gamma}(t,x)=\frac{1}{\lambda_{\gamma}^{1/2}(t)}Q_{b_{\gamma}(t),\omega_{\gamma}(t)}\bigg(\frac{x-x_{\gamma}(t)}{\lambda_{\gamma}(t)}\bigg).\label{32}
\end{gather}

We claim that 
\begin{lemma}\label{L2}The following properties hold true.
\begin{enumerate}
\item  For all $t\geq T$, we have:
\begin{gather}
\frac{\mathcal{N}_{1}^{\gamma}(t)}{\lambda^2_{\gamma}(t)}+|b_{\gamma}(t)|+\omega_{\gamma}(t)+\lambda_{\gamma}(t)\rightarrow0,\;\; \text{as }\gamma\rightarrow0,\label{33}\\
x_{\gamma}(t)\rightarrow+\infty,\;\; \text{as }\gamma\rightarrow0.\label{34}
\end{gather}
\item We have:
\begin{equation}
\label{310}
\lim_{\gamma\rightarrow0}t^*_{2,\gamma}=T.
\end{equation}
\item For all $T_0> T$, we have:
\begin{equation}
\label{35}
\tilde{u}_{\gamma}(t,x)\rightarrow v(t,x)\text{ in } \mathcal{C}([T,T_0],L^2),\;\;\text{as }\gamma\rightarrow0,
\end{equation}
\end{enumerate}
\end{lemma}
\begin{remark}\label{R1}
From the definition of $Q_{b,\omega}$, it is easy to see that for all $t\geq T$ and $R>0$,
$$\big\|Q_S^{\gamma}(t,\cdot)\big\|_{L^2(|x|<R)}\rightarrow0,\;\; \text{as }\gamma\rightarrow0,$$
which together with Lemma \ref{L2} implies \eqref{15} immediately.
\end{remark}

\begin{proof}

\noindent{\bf Step 1. Proof of \eqref{33} and \eqref{34}.}

First, we claim that 
\begin{equation}
\label{311}
t_2^*:=\liminf_{\gamma\rightarrow0}t_{2,\gamma}^*\geq T.
\end{equation}
Suppose \eqref{311} does not hold. Then there exists a $t_0<T$ and a sequence $\{\gamma_n\}$ such that $\lim_{n\rightarrow+\infty}\gamma_n=0$ and for all $n$ large enough, we have $t^*_{2,\gamma_n}<t_0<T$. From \eqref{236}, we know that
$$\lambda_{\gamma_n}(t_0)\sim \bigg(\frac{\gamma_n}{\ell^*}\bigg)^{\frac{1}{m+2}},$$
which implies that $\lim_{n\rightarrow+\infty}\lambda_{\gamma_n}(t_0)=0$. But from \eqref{239}, we have
$$\lim_{n\rightarrow+\infty}\lambda_{\gamma_n}(t_0)=\lambda(t_0)>0,$$
which leads to a contradiction.

Now we turn to the proof of 
\begin{equation}
\label{312}
\lim_{\gamma\rightarrow0}\lambda_{\gamma}(t)=0,
\end{equation}
for all $t\geq T$. From \eqref{236}, we have for all $t>t_2^*$, and $\gamma>0$ small enough,
$$\lambda_{\gamma}(t)\sim \bigg(\frac{\gamma}{\ell^*}\bigg)^{\frac{1}{m+2}},$$
hence, $\lim_{\gamma\rightarrow0}\lambda_{\gamma}(t)=0$ for all $t>t_2^*$. 

While for $t\in[T,t^*_2]$, thanks to \eqref{234}, we only need to show that $\lim_{\gamma\rightarrow0}\lambda_{\gamma}(T)=0$. Indeed, from \eqref{223} and \eqref{235} we have: for all $t\in[t^*_1,t^*_{2,\gamma}]$
\begin{equation}\label{313}
\frac{\ell^*}{3}-C\frac{\mathcal{N}_1^{\gamma}}{[\lambda_0^{\gamma}]^2}\leq -(\lambda_0^{\gamma})_t\leq 3\ell^*+C\frac{\mathcal{N}_1^{\gamma}}{[\lambda_0^{\gamma}]^2}.
\end{equation}
For all $t_0<T$ close enough to $T$, we integrate \eqref{313} from $t_0$ to $T$ using \eqref{219} and \eqref{220} to obtain
\begin{align*}
|\lambda_0^{\gamma}(T)-\lambda_0^{\gamma}(t_0)|&\lesssim \ell^*(T-t_0)+\int_{t_0}^T\frac{\mathcal{N}_1^{\gamma}}{[\lambda_0^{\gamma}]^2}\,dt\lesssim\ell^*(T-t_0)+\int_{s(t_0)}^{s(T)}\lambda_{\gamma}(s)\mathcal{N}_1^{\gamma}(s)\,ds\\
&\lesssim\ell^*(T-t_0)+\lambda_{\gamma}(t_0)\int_{s^*_1}^{+\infty}\mathcal{N}_{1}^{\gamma}(s)\,ds.
\end{align*}
Since we have for all $t\geq t_1^*$,
$$\bigg|\frac{\lambda_{\gamma}(t)}{\lambda_0^{\gamma}(t)}-1\bigg|\ll1,$$
the above inequalities imply that 
$$\lambda_{\gamma}(T)\lesssim \ell^*(T-t_0)+\lambda_{\gamma}(t_0).$$
Hence, from \eqref{238}, we have
$$\limsup_{\gamma\rightarrow0}\lambda_{\gamma}(T)\lesssim \ell^*(T-t_0)+\lambda(t_0)\lesssim \ell^*(T-t_0).$$
Let $t_0\rightarrow T$, we obtain $\lim_{\gamma\rightarrow0}\lambda_{\gamma}(T)=0$, which implies \eqref{312} immediately.

Next, from \eqref{235} and \eqref{237}, we have
\begin{equation}
\label{314}
|b_{\gamma}(t)|+\omega_{\gamma}(t)\lesssim \ell^*\lambda_{\gamma}^2(t)\rightarrow0,\quad\text{as }\gamma\rightarrow0.
\end{equation}

From \eqref{MS1}, we have $(x_{\gamma})_t\sim \lambda^{-2}_{\gamma}>0$ for all $t\geq0$. Then for all $t_0<T\leq t$, from \eqref{238} we have
\begin{equation*}
\liminf_{\gamma\rightarrow0}x_{\gamma}(t)\geq\liminf_{\gamma\rightarrow0}x_{\gamma}(t_0)=x(t_0)\sim \frac{1}{\ell_0^2(T-t_0)^2}.
\end{equation*}
Let $t_0\rightarrow T$, we obtain \eqref{34} immediately.

Now it only remains to prove
$$\lim_{\gamma\rightarrow0}\frac{\mathcal{N}_1^{\gamma}(t)}{\lambda_{\gamma}^2(t)}=0.$$ 
For all $t_0<T\leq t$, from \eqref{220}, we have
$$\frac{\mathcal{N}_1^{\gamma}(t)}{\lambda_{\gamma}^2(t)}\lesssim \frac{\mathcal{N}_1^{\gamma}(t_0)}{\lambda_{\gamma}^2(t_0)}+\frac{b_{\gamma}^2(t_0)+\omega_{\gamma}^2(t_0)}{\lambda_{\gamma}^2(t_0)}+\frac{b_{\gamma}^2(t)+\omega_{\gamma}^2(t)}{\lambda_{\gamma}^2(t)}$$
From \eqref{235}, \eqref{237} and \eqref{314}, we have
$$\frac{b_{\gamma}^2(t)+\omega_{\gamma}^2(t)}{\lambda_{\gamma}^2(t)}\rightarrow0,\quad \text{as }\gamma\rightarrow0.$$
While from (4.7), (4.12) and (4.54) in \cite{MMR1}, we have:
$$\lim_{\gamma\rightarrow0}\bigg(\frac{\mathcal{N}_1^{\gamma}(t_0)}{\lambda_{\gamma}^2(t_0)}+\frac{b_{\gamma}^2(t_0)+\omega_{\gamma}^2(t_0)}{\lambda_{\gamma}^2(t_0)}\bigg)=\frac{\mathcal{N}_1(t_0)+b^2(t_0)}{\lambda^2(t_0)}=o_{t_0\rightarrow T}(1).$$
Therefore, we obtain 
$$\lim_{\gamma\rightarrow0}\frac{\mathcal{N}_1^{\gamma}(t)}{\lambda_{\gamma}^2(t)}=0,$$ 
which concludes the proof of \eqref{33} and \eqref{34}.\\

\noindent{\bf Step 2. Proof of \eqref{310}}

Due to \eqref{311}, we only need to prove that 
\begin{equation}
\label{315}
t_{2,*}:=\limsup_{\gamma\rightarrow0}t_{2,\gamma}^*\leq T.
\end{equation}
Suppose \eqref{315} does not hold. Then there exists a $t_0>T$ and a sequence $\{\gamma_n\}$ such that $\lim_{n\rightarrow+\infty}\gamma_n=0$ and for all $n$ large enough, we have $t^*_{2,\gamma_n}>t_0>T$.

For all $\eta>0$, we integrate \eqref{313} from $T-\eta$ to $t_0$ to obtain
$$\lambda_{\gamma_n}(T-\eta)-\lambda_{\gamma_n}(t_0)\geq \frac{1}{3}\ell^*(t_0-T+\eta)-\frac{1}{100}\lambda_{\gamma_n}(T-\eta).$$
Let $n\rightarrow+\infty$, using \eqref{239} and \eqref{33}, we have
$$\lambda(T-\eta)\geq \frac{1}{10}(t_0-T+\eta),$$
for all $\eta>0$. This is a contradiction, since we have $\lim_{t\rightarrow T}\lambda(t)=0$. This concludes the proof of \eqref{310}.\\

\noindent {\bf Step 3. Proof of \eqref{35}.}

 We first introduce the following $L^2$-perturbation theory for $L^2$-critical gKdV obtained in \cite{KKSV}:
\begin{lemma}[{$L^2$-perturbation theory, Theorem 3.1 in \cite{KKSV}}]\label{L1}
Let $I$ be an interval containing $0$ and $w\in\mathcal{C}(I,L^2)$ is a solution to \eqref{CP} on $I$ with
$$\|w\|_{L_t^{\infty}L^2_x(I\times\mathbb{R})}+\|w\|_{L_x^{5}L_t^{10}(\mathbb{R}\times I)}<M.$$
Suppose $\tilde{w}\in\mathcal{C}(I,L^2)$ is a solution to the following equation:
$$\partial_t\tilde{w}+\tilde{w}_{xxx}+(\tilde{w}^5)_x=e,$$
with
\begin{align*}
\|w(0)-\tilde{w}(0)\|_{L^2}\leq M',
\end{align*}
and
\begin{align*}
\Big\|e^{-t\partial_x^3}\big(w(0)-\tilde{w}(0)\big)\Big\|_{L_x^5L_t^{10}(\mathbb{R}\times I)}+\|e\|_{L_x^{5/4}L_t^{10/9}(\mathbb{R}\times I)}<\epsilon,
\end{align*}
for some $M'>0$, some $0<\epsilon<\epsilon_0(M,M')$. 
Then we have:
\begin{equation}
\label{39}
\|w-\tilde{w}\|_{L_t^{\infty}L^2_x(I\times\mathbb{R})}+\|w-\tilde{w}\|_{L_x^{5}L_t^{10}(\mathbb{R}\times I)}<C_0(M,M')\epsilon.
\end{equation}
\end{lemma}
\begin{remark}
The statement of Lemma \ref{L1} is slightly different from Theorem 3.1 in \cite{KKSV}, but the proof is exactly the same. We omit the proof here.
\end{remark}

Now we turn to the proof of \eqref{35}. For the remainder term $\tilde{u}_{\gamma}(t)$ with $t\geq T$, direct computation leads to 
\begin{equation*}
\partial_t \tilde{u}_{\gamma}+\big[(\tilde{u}_{\gamma})_{xx}+\tilde{u}_{\gamma}^5\big]_x=-\mathcal{E}-\big[F_1(\tilde{u}_{\gamma})\big]_x-\big[F_2(\tilde{u}_{\gamma})\big]_x,
\end{equation*}
where
\begin{align*}
\mathcal{E}=\frac{1}{\lambda_{\gamma}^{7/2}(t)}\bigg[&-\Psi_{b_{\gamma},\omega_{\gamma}}-(b_{\gamma})_s\frac{\partial Q_{b_{\gamma},\omega_{\gamma}}}{\partial b_{\gamma}}-(\omega_{\gamma})_s\frac{\partial Q_{b_{\gamma},\omega_{\gamma}}}{\partial \omega_{\gamma}}\\
&-\bigg(\frac{\partial_s\lambda_{\gamma}}{\lambda_{\gamma}}+b_{\gamma}\bigg)\Lambda Q_{b_{\gamma},\omega_{\gamma}}-\bigg(\frac{\partial_sx_{\gamma}}{\lambda_{\gamma}}-1\bigg)Q_{b_{\gamma},\omega_{\gamma}}'\bigg]\bigg(\frac{x-x_{\gamma}(t)}{\lambda_{\gamma}(t)}\bigg),
\end{align*}
and
\begin{align*}
&F_1(\tilde{u}_{\gamma})=(Q_S^{\gamma}+\tilde{u}_{\gamma})^5-[Q_S^{\gamma}]^5-\tilde{u}_{\gamma}^5\\
&F_2(\tilde{u}_{\gamma})=-\gamma\big[(Q_S^{\gamma}+\tilde{u}_{\gamma})|Q_S^{\gamma}+\tilde{u}_{\gamma}|^{q-1}-Q_S^{\gamma}|Q_S^{\gamma}|^{q-1}\big].
\end{align*}

For all $\eta>0$ small enough, if $0<\gamma<\gamma(\eta)$ is small enough, we have
\begin{align}
\label{316}
&\|\tilde{u}_{\gamma}(T-\eta)-v(T-\eta)\|_{L^2} \nonumber\\
&\leq \|\tilde{u}_{\gamma}(T-\eta)-\tilde{u}(T-\eta)\|_{L^2}+\|\tilde{u}(T-\eta)-u^*\|_{L^2}+\|v(T-\eta)-u^*\|_{L^2}\nonumber\\
&\lesssim \delta(\eta),
\end{align}
where we use \eqref{224} and the fact that $v(T)=u^*$.

We claim that for all $T_0>T$ and $0<\eta<\eta(T_0)$ small enough, there exists $\gamma(\eta)>0$ such that if $0<\gamma<\gamma(\eta)$, then we have
\begin{equation}
\big\|\mathcal{E}\big\|_{L^{5/4}_xL^{10/9}_t(\mathbb{R}\times[T-\eta,T_0])}+\sum_{i=1}^2\|\partial_xF_i(\tilde{u}_{\gamma})\|_{L_x^{5/4}L_t^{10/9}(\mathbb{R}\times [T-\eta,T_0])}\lesssim_{T_0} \delta(\eta),\label{317}
\end{equation}

Then, we can apply Lemma \ref{L1} to $\tilde{u}_{\gamma}$ and $v$ on the interval $[T-\eta,T_0]$, using \eqref{226}, \eqref{316} and \eqref{317} to obtain
\begin{equation}
\label{318}
\sup_{t\in[T-\eta,T_0]}\|\tilde{u}_{\gamma}(t)-v(t)\|_{L^2}\lesssim_{T_0}\delta(\eta),
\end{equation}
provided that $0<\gamma<\gamma(\eta)$. It is easy to see that \eqref{318} implies \eqref{35} immediately.

Now, it only remains to prove \eqref{317}. First, from \eqref{23}, \eqref{MS1} and \eqref{MS2} we have%
\footnote{Recall that we take $\beta=\frac{3}{4}$ in \eqref{203}.}%
:
\begin{align}
&\big\|\mathcal{E}\big\|^{10/9}_{L^{5/4}_xL^{10/9}_t(\mathbb{R}\times[T-\eta,T_0])}\lesssim \big\|\mathcal{E}\big\|^{10/9}_{L_t^{10/9}L^{5/4}_x([T-\eta,T_0]\times\mathbb{R})}\nonumber\\
&\lesssim\int_{T-\eta}^{T_0}\frac{1}{\lambda_{\gamma}^3(t)}\bigg[\big\|\Psi_{b_{\gamma},\omega_{\gamma}}\big\|_{L^{\frac{5}{4}}}+|(b_{\gamma})_s|\bigg\|\frac{\partial Q_{b_{\gamma},\omega_{\gamma}}}{\partial b_{\gamma}}\bigg\|_{L^{\frac{5}{4}}}+|(\omega_{\gamma})_s|\bigg\|\frac{\partial Q_{b_{\gamma},\omega_{\gamma}}}{\partial \omega_{\gamma}}\bigg\|_{L^{\frac{5}{4}}}\nonumber\\
&\qquad+\bigg|\frac{\partial_s\lambda_{\gamma}}{\lambda_{\gamma}}+b_{\gamma}\bigg|\|\Lambda Q_{b_{\gamma},\omega_{\gamma}}\|_{L^{\frac{5}{4}}}+\bigg|\frac{\partial_sx_{\gamma}}{\lambda_{\gamma}}-1\bigg|\|Q'_{b_{\gamma},\omega_{\gamma}}\|_{L^{\frac{5}{4}}}\bigg]^{\frac{10}{9}}\,dt\nonumber\\
&\lesssim\int_{T-\eta}^{T_0}\frac{1}{\lambda_{\gamma}^3(t)}\Big[(\mathcal{N}_{1,\rm loc}^{\gamma})^{1/2}+|b_{\gamma}|^{\frac{23}{20}}+\omega_{\gamma}|b_{\gamma}|\Big]^{10/9}\,dt\nonumber\\
&\lesssim_{T_0}\int_{T-\eta}^{T_0}\frac{(\omega_{\gamma}|b_{\gamma}|)^{10/9}}{\lambda_{\gamma}^3(t)}\,dt+\int_{T-\eta}^{T}\frac{|b_{\gamma}|^{23/18}}{\lambda_{\gamma}^3(t)}\,dt+\bigg(\int_{T}^{T_0}\frac{|b_{\gamma}|^{115/54}}{\lambda_{\gamma}^5(t)}\,dt\bigg)^{3/5}\nonumber\\
&\quad+\bigg(\int_{T-\eta}^{T_0}\frac{\mathcal{N}^{\gamma}_{1,\rm loc}}{\lambda_{\gamma}^{27/5}(t)}\,dt\bigg)^{5/9}.\label{319}
\end{align}

We estimate all these terms separately. First, from \eqref{235}--\eqref{237}, we have
\begin{multline}
\label{320}
\int_{T-\eta}^{T_0}\frac{(\omega_{\gamma}|b_{\gamma}|)^{10/9}}{\lambda_{\gamma}^3(t)}\,dt\lesssim \int_{T-\eta}^{T_0}(\ell^*)^{3/2}\big[\omega_{\gamma}(t)\big]^{13/18}\,dt\\
\lesssim (T_0-T+\eta)(\ell^*)^{\frac{20m+27}{9(m+2)}}\gamma^{\frac{13}{9(m+2)}}\lesssim_{T_0}\delta(\eta).
\end{multline}

Next, we integrate \eqref{313} from $t$ to $T$ to obtain
$$\lambda_{\gamma}(t)\gtrsim \lambda_{\gamma}(T)+\ell^*(T-t)\gtrsim\ell^*(T-t),$$
for all $t\in[t^*_1,T)$. Together with \eqref{235} and \eqref{237}, we have:
\begin{equation}
\label{321}
\int_{T-\eta}^{T}\frac{|b_{\gamma}|^{\frac{23}{18}}(t)}{\lambda_{\gamma}^3(t)}\,dt\lesssim (\ell^*)^{\frac{23}{18}}\int_{T-\eta}^T\frac{dt}{\lambda_{\gamma}^{4/9}(t)}\lesssim \int_{T-\eta}^T\frac{dt}{\big[\ell^*(T-t)\big]^{4/9}}\lesssim \delta(\eta).
\end{equation}

Then from \eqref{221}, \eqref{236}, \eqref{33} and \eqref{310}, we have:
\begin{align}
\label{322}
\int_{T}^{T_0}\frac{|b_{\gamma}|^{\frac{115}{54}}(t)}{\lambda_{\gamma}^5(t)}\,dt&\lesssim \bigg(\sup_{t\in[T,T_0]}\frac{|b_{\gamma}(t)|^{7/54}}{\lambda_{\gamma}^2(t)}\bigg)\times\int_{s(T)}^{s(T_0)}b_{\gamma}^2(s)\,ds\nonumber\\
&\lesssim\bigg(\sup_{t\in[T,T_0]}\frac{|b_{\gamma}(t)|^{7/54}}{\lambda_{\gamma}^2(t)}\bigg)\times\big[\mathcal{N}_1(T)+\omega_{\gamma}(T)+\sup_{t\in[T,T_0]}|b_{\gamma}(t)|\big]\nonumber\\
&\lesssim\bigg(\sup_{t\in[T,T_0]}|b_{\gamma}(t)|^{7/54}\bigg)\times\bigg(\frac{\mathcal{N}_1^{\gamma}(T)+\omega_{\gamma}(T)}{\lambda_{\gamma}^2(T)}+\sup_{t\in[T,T_0]}\frac{|b_{\gamma}(t)|}{\lambda_{\gamma}^2(t)}\bigg)\nonumber\\
&\lesssim \delta(\eta),
\end{align}
provided that $\gamma<\gamma(\eta)$ small enough. 

Finally, for the term
$$\int_{T-\eta}^{T_0}\frac{\mathcal{N}^{\gamma}_{1,\rm loc}(t)}{\lambda_{\gamma}^{27/5}(t)}\,dt,$$
from \eqref{234}, we have:
\begin{equation}
\label{323}
\int_{T-\eta}^{T_0}\frac{\mathcal{N}^{\gamma}_{1,\rm loc}(t)}{\lambda_{\gamma}^{27/5}(t)}\,dt\lesssim \lambda_{\gamma}^{\frac{3}{5}}(T-\eta)\int_{s^*_1}^{+\infty}\frac{\mathcal{N}^{\gamma}_{1,\rm loc}(s)}{\lambda^3_{\gamma}(s)}\,ds.
\end{equation}
We claim that 
\begin{equation}
\label{324}
\int_{s^*_1}^{+\infty}\frac{\mathcal{N}^{\gamma}_{1,\rm loc}(s)}{\lambda^3_{\gamma}(s)}\,ds\lesssim 1.
\end{equation}

Indeed, from \eqref{MF1}, we have for all $s\in[s^*_1,+\infty)$:
\begin{equation}
\label{325}
\lambda_{\gamma}^3\bigg(\frac{\mathcal{F}^{\gamma}_{2,1}}{\lambda_{\gamma}^3}\bigg)_s+\mu\int\Big((\partial_y\varepsilon_{\gamma})^2+\varepsilon_{\gamma}^2\Big)\varphi'_{2,B}\lesssim_B b_{\gamma}^2(\omega_{\gamma}^2+b_{\gamma}^2)-3\frac{(\lambda_{\gamma})_s}{\lambda_{\gamma}}\mathcal{F}^{\gamma}_{2,1}.
\end{equation}
Recall from (3.21) in \cite{L3}, we have for all $s\in[s^*_1,+\infty)$:
$$\int_{y>0}y^2\varepsilon_{\gamma}^2(s)\lesssim\bigg(1+\frac{1}{\lambda_{\gamma}^{\frac{10}{9}}(s)}\bigg)\big[\mathcal{N}^{\gamma}_{2,\rm loc}(s)\big]^{\frac{8}{9}}.$$
Together with \eqref{MS1}, we have for all $s\in[s^*_1,+\infty)$: 
\begin{align*}
\bigg|\frac{(\lambda_{\gamma})_s}{\lambda_{\gamma}}\mathcal{F}^{\gamma}_{2,1}\bigg|&\lesssim\Big(|b_{\gamma}|+\big[\mathcal{N}^{\gamma}_{1,\rm loc}\big]^{\frac{1}{2}}\Big)\Bigg[\bigg(1+\frac{1}{\lambda_{\gamma}^{\frac{10}{9}}(s)}\bigg)\big[\mathcal{N}^{\gamma}_{2,\rm loc}\big]^{\frac{8}{9}}+\int(\varepsilon_{\gamma})^2_{y}\psi_B\Bigg]\\
&\lesssim b_{\gamma}^4+\delta(\kappa)\int\big(\varepsilon_{\gamma}^2+(\varepsilon_{\gamma})^2_{y}\big)\varphi_{2,B}'.
\end{align*}
Injecting the above estimate into \eqref{325} and integrating from $s^*_1$ to $+\infty$, we obtain:
\begin{equation}\label{326}
\int_{s^*_1}^{+\infty}\frac{1}{\lambda_{\gamma}^3}\bigg(\int\Big((\partial_y\varepsilon_{\gamma})^2+\varepsilon_{\gamma}^2\Big)\varphi'_{2,B}\bigg)\lesssim \frac{\mathcal{N}^{\gamma}_2(s^*_1)}{\lambda^3_{\gamma}(s_1^*)}+\int_{s_1^*}^{+\infty}\frac{b_{\gamma}^2(\omega_{\gamma}^2+b_{\gamma}^2)}{\lambda_{\gamma}^3},
\end{equation} 
where we use \eqref{COER} for the above inequality. From \eqref{221}, \eqref{235}, \eqref{237} and \eqref{240}, we have:
$$\int_{s_1^*}^{+\infty}\frac{b_{\gamma}^2(\omega_{\gamma}^2+b_{\gamma}^2)}{\lambda_{\gamma}^3}\lesssim (\ell^*)^2\int_{s_1^*}^{\infty}b_{\gamma}^2\ll1.$$
From \eqref{230}, we have
$$\frac{\mathcal{N}^{\gamma}_2(s^*_1)}{\lambda^3_{\gamma}(s_1^*)}\lesssim \mathcal{N}^{\gamma}_2(s^*_1)\ll1.$$
Hence, \eqref{326} implies \eqref{324} immediately.

Combining \eqref{319}--\eqref{324}, we have:
\begin{equation}
\label{327}
\big\|\mathcal{E}\big\|_{L^{5/4}_xL^{10/9}_t(\mathbb{R}\times[T-\eta,T_0])}\lesssim_{T_0}\delta(\eta),
\end{equation}
provided that $0<\gamma<\gamma(\eta)$ small enough.

Next, we estimate
$$\|\partial_xF_1(\tilde{u}_{\gamma})\|_{L_x^{5/4}L_t^{10/9}(\mathbb{R}\times [T-\eta,T_0])}.$$
Direct computation leads to 
\begin{align*}
&\|\partial_xF_1(\tilde{u}_{\gamma})\|^{10/9}_{L_x^{5/4}L_t^{10/9}(\mathbb{R}\times [T-\eta,T_0])}\leq \|\partial_xF_1(\tilde{u}_{\gamma})\|^{10/9}_{L_t^{10/9}L_x^{5/4}([T-\eta,T_0]\times\mathbb{R})}\\
&\lesssim\int_{T-\eta}^{T_0}\frac{1}{\lambda_{\gamma}^3(t)}\bigg[\Big\|(\varepsilon_{\gamma})_y\big[(\varepsilon_{\gamma}+Q_{b_{\gamma},\omega_{\gamma}})^4-(\varepsilon_{\gamma})^4\big]\Big\|_{L_x^{\frac{5}{4}}}\\
&\qquad+\Big\|(Q_{b_{\gamma},\omega_{\gamma}})_y\big[(\varepsilon_{\gamma}+Q_{b_{\gamma},\omega_{\gamma}})^4-Q^4_{b_{\gamma},\omega_{\gamma}}\big]\Big\|_{L_x^{\frac{5}{4}}}\bigg]^{\frac{10}{9}}\,dt.\\
&\lesssim\int_{T-\eta}^{T_0}\frac{1}{\lambda_{\gamma}^3(t)}\Big[\big\|(\varepsilon_{\gamma})_yQ^4_{b_{\gamma},\omega_{\gamma}}\big\|_{L_x^{\frac{5}{4}}}+\big\|(\varepsilon_{\gamma})_y\varepsilon_{\gamma}^3Q_{b_{\gamma},\omega_{\gamma}}\big\|_{L_x^{\frac{5}{4}}}\\
&\qquad+\big\|\varepsilon_{\gamma}^4(Q_{b_{\gamma},\omega_{\gamma}})_y\big\|_{L_x^{\frac{5}{4}}}+\big\|\varepsilon_{\gamma}Q^3_{b_{\gamma},\omega_{\gamma}}(Q_{b_{\gamma},\omega_{\gamma}})_y\big\|_{L_x^{\frac{5}{4}}}\Big]^{\frac{10}{9}}\,dt
\end{align*}
Thus from \eqref{201} and \eqref{202}, we have:
\begin{align}
&\|\partial_xF_1(\tilde{u}_{\gamma})\|^{10/9}_{L_x^{5/4}L_t^{10/9}(\mathbb{R}\times [T-\eta,T_0])}\nonumber\\
&\lesssim\int_{T-\eta}^{T_0}\frac{1}{\lambda_{\gamma}^3(t)}\Bigg[\bigg(\int\big[(\varepsilon_{\gamma})^2_y+\varepsilon^2_{\gamma}\big]e^{-\frac{|y|}{10}}\bigg)^{\frac{1}{2}}+|b_{\gamma}|^{4-\frac{4\beta}{5}}\|\varepsilon_{\gamma}\|_{H^1}\nonumber\\
&\qquad+|b_{\gamma}|^{1-\frac{3\beta}{10}}\big(\|\varepsilon_{\gamma}\|_{L^{\infty}}^3\|\varepsilon_{\gamma}\|_{L^2}+\|(\varepsilon_{\gamma})_y\|_{L^2}\|\varepsilon_{\gamma}\|^3_{L^{\infty}}\big)\Bigg]^{\frac{10}{9}}\,dt.\label{328}
\end{align}
From \eqref{MC} and \eqref{EC}, we have for all $s\in[0,+\infty)$
$$\|\varepsilon_{\gamma}(s)\|_{H^1}+\frac{\|(\varepsilon_{\gamma})_y(s)\|_{L^2}}{\lambda_{\gamma}(s)}\lesssim \delta(\alpha_0)\ll1.$$
Together with \eqref{235}, \eqref{237}, \eqref{33}, and the fact that $\beta=\frac{3}{4}$, we have:
\begin{align}
\label{329}
&\int_{T-\eta}^{T_0}\frac{1}{\lambda_{\gamma}^3(t)}\Big[|b_{\gamma}|^{4-\frac{4\beta}{5}}\big(\|\varepsilon_{\gamma}\|_{H^1}\big)+|b_{\gamma}|^{1-\frac{3\beta}{10}}\|\varepsilon_{\gamma}\|_{L^{\infty}}^3\|\varepsilon_{\gamma}\|_{H^1}\Big]^{\frac{10}{9}}\,dt\nonumber\\
&\lesssim\int_{T-\eta}^{T_0}\frac{1}{\lambda_{\gamma}^3(t)}\Big[|b_{\gamma}(t)|^{\frac{34}{9}}+|b_{\gamma}(t)|^{\frac{31}{36}}\big\|(\varepsilon_{\gamma})_y\big\|_{L^2}^{\frac{5}{3}}\Big]\,dt\nonumber\\
&\lesssim \int_{T-\eta}^{T_0}\Big[(\ell^*)^{\frac{3}{2}}|b_{\gamma}(t)|^{\frac{41}{18}}+(\ell^*)^{\frac{2}{3}}|b_{\gamma}(t)|^{\frac{7}{36}}\Big]\,dt\lesssim_{T_0} \delta(\eta),
\end{align}
provided that $0<\gamma<\gamma(\eta)$ small enough. Then, from \eqref{323} and \eqref{324}, we have
\begin{equation}
\label{330}
\int_{T-\eta}^{T_0}\frac{1}{\lambda_{\gamma}^3(t)}\bigg(\int\big[(\varepsilon_{\gamma})^2_y+\varepsilon^2_{\gamma}\big]e^{-\frac{|y|}{10}}\bigg)^{\frac{1}{2}}\lesssim_{T_0}\Bigg[\int_{T-\eta}^{T_0}\frac{\mathcal{N}^{\gamma}_{1,\rm loc}(t)}{\lambda_{\gamma}^{27/5}(t)}\,dt\Bigg]^{\frac{5}{9}}\lesssim_{T_0}\delta(\eta).
\end{equation}
Combining \eqref{328}--\eqref{330}, we obtain that
\begin{equation}
\label{331}
\|\partial_xF_1(\tilde{u}_{\gamma})\|_{L_x^{5/4}L_t^{10/9}(\mathbb{R}\times [T-\eta,T_0])}\lesssim_{T_0}\delta(\eta).
\end{equation}

Finally, we estimate the term
$$\|\partial_xF_2(\tilde{u}_{\gamma})\|_{L_x^{5/4}L_t^{10/9}(\mathbb{R}\times [T-\eta,T_0])}.$$
Following from similar arguments, we have
\begin{align}
\label{335}
&\|\partial_xF_2(\tilde{u}_{\gamma})\|^{10/9}_{L_x^{5/4}L_t^{10/9}(\mathbb{R}\times [T-\eta,T_0])}\leq \|\partial_xF_2(\tilde{u}_{\gamma})\|^{10/9}_{L_t^{10/9}L^{5/4}_x([T-\eta,T_0]\times\mathbb{R})}\nonumber\\
&\lesssim\gamma^{10/9}\int_{T-\eta}^{T_0}\frac{1}{[\lambda_{\gamma}(t)]^{(5q+2)/9}}\bigg[\Big\|(Q_{b_{\gamma},\omega_{\gamma}})_y\big[|\varepsilon_{\gamma}+Q_{b_{\gamma},\omega_{\gamma}}|^{q-1}-|Q_{b_{\gamma},\omega_{\gamma}}|^{q-1}\big]\Big\|_{L_x^{\frac{5}{4}}}\nonumber\\
&\qquad+\Big\|(\varepsilon_{\gamma})_y\big|\varepsilon_{\gamma}+Q_{b_{\gamma},\omega_{\gamma}}\big|^{q-1}\Big\|_{L_x^{\frac{5}{4}}}\bigg]^{\frac{10}{9}}\,dt\nonumber\\
&\lesssim\gamma^{10/9}\int_{T-\eta}^{T_0}\frac{1}{[\lambda_{\gamma}(t)]^{(5q+2)/9}}\bigg[|b_{\gamma}|^{1-\frac{4\beta}{5}}\big\||\varepsilon_{\gamma}|^{q-1}\big\|_{L_x^{\infty}}+[\mathcal{N}_{1,\rm loc}^{\gamma}]^{\frac{1}{2}}+|b_{\gamma}|^{q-1-\frac{3\beta}{10}}\nonumber\\
&\qquad+\|(\varepsilon_{\gamma})_y\|_{L^2}\big\||\varepsilon_{\gamma}|^{q-1}\big\|_{L_x^{\frac{10}{3}}}\bigg]^{\frac{10}{9}}\,dt.
\end{align}
We use \eqref{MC}, \eqref{EC} and the Sobolev embedding
$$\|\varepsilon_{\gamma}\|^2_{L^{\infty}}\lesssim\|\varepsilon_{\gamma}\|_{L^2}\|(\varepsilon_{\gamma})_y\|_{L^2}$$
again to estimate:
\begin{align}
\label{332}
&\int_{T-\eta}^{T_0}\frac{\gamma^{10/9}}{[\lambda_{\gamma}(t)]^{\frac{5q+2}{9}}}\bigg[|b_{\gamma}|^{1-\frac{4\beta}{5}}\big\||\varepsilon_{\gamma}|\big\|^{q-1}_{L_x^{\infty}}+|b_{\gamma}|^{q-1-\frac{3\beta}{10}}+\|(\varepsilon_{\gamma})_y\|_{L^2}\big\||\varepsilon_{\gamma}|^{q-1}\big\|_{L_x^{\frac{10}{3}}}\bigg]^{\frac{10}{9}}\,dt.\nonumber\\
&\lesssim\int_{T-\eta}^{T_0}\frac{\gamma^{10/9}}{[\lambda_{\gamma}(t)]^{\frac{5q+2}{9}}}\Big[[\lambda_{\gamma}(t)]^{\frac{5q-3}{10}}+[\lambda_{\gamma}(t)]^{\frac{40q-49}{20}}+[\lambda_{\gamma}(t)]^{\frac{5q+2}{10}}\Big]^{\frac{10}{9}}\,dt\nonumber\\
&\lesssim_{T_0}\delta(\eta),
\end{align}
provided that $0<\gamma<\gamma(\eta)$ is small enough.

Now it only remains to estimate
$$\gamma^{10/9}\int_{T-\eta}^{T_0}\frac{[\mathcal{N}_{1,\rm loc}^{\gamma}(t)]^{\frac{5}{9}}}{[\lambda_{\gamma}(t)]^{(5q+2)/9}}\,dt.$$
Recall that 
$$\omega_{\gamma}(t)=\frac{\gamma}{\lambda_{\gamma}^{\frac{q-5}{2}}(t)}.$$
Together with \eqref{221} and \eqref{234}--\eqref{237}, we have 
\begin{align}
\label{333}
&\gamma^{10/9}\int_{T-\eta}^{T_0}\frac{[\mathcal{N}_{1,\rm loc}^{\gamma}(t)]^{\frac{5}{9}}}{[\lambda_{\gamma}(t)]^{(5q+2)/9}}\,dt\lesssim\int_{T-\eta}^{T_0}[\omega_{\gamma}(t)]^{10/9}\frac{[\mathcal{N}_{1,\rm loc}^{\gamma}(t)]^{\frac{5}{9}}}{[\lambda_{\gamma}(t)]^{3}}\,dt\nonumber\\
&\lesssim\int_{T-\eta}^{T_0}\frac{[\mathcal{N}_{1,\rm loc}^{\gamma}(t)]^{\frac{5}{9}}}{[\lambda_{\gamma}(t)]^{7/9}}\,dt\lesssim_{T_0}\Bigg[\int_{T-\eta}^{T_0}\frac{\mathcal{N}_{1,\rm loc}^{\gamma}(t)}{[\lambda_{\gamma}(t)]^{7/5}}\,dt\Bigg]^{\frac{5}{9}}\nonumber\\
&\lesssim_{T_0} [\lambda_{\gamma}(T-\eta)]^{\frac{8}{9}}\Bigg[\int_{s^*_1}^{+\infty}\mathcal{N}_{1,\rm loc}^{\gamma}(s)\,ds\Bigg]^{\frac{5}{9}}\lesssim_{T_0}\delta(\eta),
\end{align}
provided that $0<\gamma<\gamma(\eta)$ is small enough.

Combining \eqref{335}--\eqref{333}, we have
\begin{equation}
\label{334}
\|\partial_xF_2(\tilde{u}_{\gamma})\|_{L_x^{5/4}L_t^{10/9}(\mathbb{R}\times [T-\eta,T_0])}\lesssim_{T_0}\delta(\eta),
\end{equation}
which together with \eqref{327} and \eqref{331} implies \eqref{317} immediately.

Therefore we conclude the proof of \eqref{35} hence, the proof of Lemma \ref{L2}. 
\end{proof}

Recall from Remark \ref{R1}, we complete the proof of \eqref{15}. 

Now it only remains to prove \eqref{16}.  From the definition of $u_{\rm ext}(t)$, it is easy to see that \eqref{16} holds true for all $t<T$. If $t\geq T$, from Lemma \ref{L2}, we have:
\begin{equation}\label{336}
Q^{\gamma}_S(t,\cdot)\rightarrow 0\;\text{in }H^1_{\rm loc},\;\text{as }\gamma\rightarrow0^+.
\end{equation}
For all $z(t,x)\in\mathcal{C}_0^{\infty}(\mathbb{R}\times\mathbb{R})$, injecting \eqref{35}, \eqref{336} into the following equation 
\begin{align*}
&\int_{\mathbb{R}}u_{\gamma}(t,x)z(t,x)\,dx-\int_{\mathbb{R}}u_0(x)z(0,x)\,dx\nonumber\\
&=\int_0^t\bigg\{\int_{\mathbb{R}}u_{\gamma}(s,x)\partial_tz(s,x)\,dx+\int_{\mathbb{R}}u_{\gamma}(s,x)\partial_x^3z(s,x)\,dx\nonumber\\
&\qquad+\int_{\mathbb{R}}\bigg[u^5_{\gamma}(s,x)-\gamma u_{\gamma}|u_{\gamma}|^{q-1}(s,x)\bigg]\partial_xz(s,x)\,dx\bigg\}\,ds,
\end{align*}
we obtain \eqref{16} immediately, which concludes the proof of Theorem \ref{MT}.

\bibliographystyle{amsplain}
\bibliography{ref}
\end{document}